\documentclass[12pt]{amsart}
\usepackage[margin=0.9in]{geometry}%,includeheadfoot-----------------------------------------changed by Wu
\usepackage{amsmath}
\usepackage{slashed}
\usepackage{graphicx}
\usepackage[T1]{fontenc}
\usepackage[utf8]{inputenc}

\usepackage{microtype}
\usepackage{amssymb,amsthm,amsfonts}
\usepackage{thmtools, thm-restate}

\usepackage{dsfont}
\usepackage{mathrsfs}

\usepackage{tikz-cd}
\usepackage{enumitem}

\usepackage{hyperref}
\usepackage{mathtools}
\mathtoolsset{showonlyrefs}

\declaretheorem[name=Remark, numbered=no, style=remark]{remark}%----------------------------------------added by Wu
\declaretheorem[name=Example, style=remark]{ex}%--------------------------------------------------------added by Wu

\newtheorem{thm}{Theorem}[section]

\newtheorem{prop}[thm]{Proposition}
\newtheorem{lemma}[thm]{Lemma}
\theoremstyle{definition}
\def\R{\mathbb R}

\newcounter{stepnum}

\DeclareMathOperator{\Rm}{R}
\DeclareMathOperator{\diverg}{div}

\DeclareMathOperator{\End}{End}

\DeclareMathOperator{\supp}{supp}
\DeclareMathOperator{\Ric}{Ric}

\newcommand{\dd}{\mathop{}\!\mathrm{d}}
\newcommand{\A}{\mathbb{A}}
\newcommand{\D}{\slashed{D}}
\newcommand{\p}{\partial}
\newcommand{\al}{\alpha}
\newcommand{\be}{\beta}
\newcommand{\na}{\nabla}
\newcommand{\pd}{\slashed{\partial}}

\newcommand{\vep}{\varepsilon}

\newcommand{\dv}{\dd vol}

\newcommand{\dt}[1]{\frac{\dd}{\dd t}\Big|_{t=#1}}

\newcommand{\sph}{\mathbb{S}}
\def\bee{\begin{eqnarray}}
\def\beee{\begin{eqnarray*}}
\def\eee{\end{eqnarray}}
\def\eeee{\end{eqnarray*}}

\def\ba{\begin{array}}
\def\ea{\end{array}}

\def\R{\mathbb R}

\title[Energy quantization]{Energy quantization for a nonlinear sigma model\\ with critical gravitinos}

\begin{document}

\author{J\"{u}rgen Jost, Ruijun Wu and Miaomiao Zhu}

\address{Max Planck Institute for Mathematics in the Sciences\\Inselstr. 22--26\\D-04103 Leipzig, Germany}
	\email{jjost@mis.mpg.de}

\address{Max Planck Institute for Mathematics in the Sciences\\Inselstr. 22--26\\D-04103 Leipzig, Germany}
	\email{Ruijun.Wu@mis.mpg.de}

\address{School of Mathematical Sciences, Shanghai Jiao Tong University\\Dongchuan Road 800\\200240 Shanghai, P.R.China}
	\email{mizhu@sjtu.edu.cn}
	
\thanks{The third author was supported in part by National Science Foundation of China (No. 11601325).}

\date{\today}

\begin{abstract}
 We study some analytical and geometric properties of a two-dimensional nonlinear sigma model with gravitino which comes from supersymmetric string theory. When the action is critical w.r.t. variations of the various fields including the gravitino, there is a symmetric, traceless and divergence-free energy-momentum tensor, which gives rise to a holomorphic quadratic differential. Using it we obtain a Pohozaev type identity and finally we can establish the energy identities along a weakly convergent sequence of fields with uniformly bounded energies.
\end{abstract}


\maketitle

%%%%%%%%%%%%%%%%%%%%%%%%%%%%%%%%%%%%%%%%%%%----1----%%%%%%%%%%%%%%%%%%%%%%%%%%%%%%%%%%%%%%%%%%%%%%%%%%%%%%%%%%%%%%%%%%%%%%
\section{Introduction}

The 2-dimensional nonlinear sigma models constitute important models in  quantum field theory. They have not only physical applications, but also geometric implications, and therefore their properties have been the focus of important lines of research. In mathematics, they arise as two-dimensional harmonic maps and pseudo holomorphic curves. In modern physics the basic matter fields are described by vector fields as well as spinor fields, which are coupled by  supersymmetries. The base manifolds are two-dimensional, and therefore their  conformal and  spin structures come into play. From the physics side, in the 1970s  a supersymmetric 2-dimensional nonlinear sigma model was proposed in \cite{brink1976locally, deser1976complete}; the name ``supersymmetric'' comes from the fact that the action functional is invariant under certain transformations of the matter fields, see for instance  \cite{deligne1999quantum, jost2009geometry}. From the perspective of geometric analysis, they seem to be natural candidates for   a variational approach, and one might expect that the powerful variational methods developed for harmonic maps and pseudo holomorphic curves could be applied here as well. However, because of the various spinor fields involved, new difficulties arise. The geometric aspects have been developed in mathematical terms in   \cite{kessler2014functional}, but this naturally involves anti-commuting variables which are not amenable to inequalities, and therefore variational methods cannot be applied, and one rather needs algebraic tools. This would lead to what one may call  super harmonic maps. Here, we adopt a different approach. We transform the anti-commuting variables into  commuting ones, as in ordinary Riemannian geometry. In particular, the domains of the action functionals are ordinary Riemann surfaces instead of super Riemann surfaces. Then one has more fields to control, not only the maps between Riemannian manifolds and Riemannian metrics, but also their super partners. Such a model was developed and investigated in \cite{jost2016regularity}. Part of the symmetries, including some super symmetries, are inherited, although some essential supersymmetries are hidden or  lost. As is known, the symmetries of such functionals are quite important for the analysis, in order to overcome some analytical problems that arise as we are working in a limiting situation of the Palais-Smale condition. Therefore, here we shall develop a setting with a large symmetry group. This will enable us to carry out the essential steps of the variational analysis. The analytical key will be a Pohozaev type identity.

We will follow the notation conventions of \cite{jost2016regularity}, which are briefly recalled in the following. Let $(M,g)$ be an oriented closed Riemannian surface with a fixed spin structure, and let $S\to M$ be a spinor bundle, of real rank four, associated to the given spin structure. Note that the Levi-Civita connection $\na^M$ on $M$ and the Riemannian metric $g$ induce a spin connection $\na^s$ on $S$ in a canonical way and a spin metric $g_s$ which is a fiberwise real inner product\footnote{Here we take the real  rather than the Hermitian one used in some previous works on Dirac-harmonic maps (with or without curvature term), as clarified in \cite{jost2016regularity}.} , see \cite{lawson1989spin, jost2011riemannian}. The spinor bundle $S$ is a left module over the Clifford bundle $\operatorname{Cl}(M,-g)$ with the Clifford map being denoted by $\gamma\colon TM\to \End(S)$; sometimes  it will be simply denoted by a dot. The Clifford relation reads
\begin{equation}
 \gamma(X)\gamma(Y)+\gamma(Y)\gamma(X)=-2g(X,Y), \qquad \forall X,Y\in\mathscr{X}(M).
\end{equation}
The Clifford action is compatible with the spinor metric and the spin connection, making $S$ into a Dirac bundle in the sense of \cite{lawson1989spin}. Therefore, the bundle $S\otimes TM$ is also a Dirac bundle over $M$, and a section $\chi\in \Gamma(S\otimes TM)$ is taken as a super partner of the Riemannian metric, and called a gravitino. The Clifford multiplication gives rise to a map $\delta_\gamma\colon S\otimes TM\to S$, where $\delta_\gamma(s\otimes v)=\gamma(v)s=v\cdot s$ for $s\in\Gamma(S)$ and $v\in \Gamma(TM)$, and extending linearly. This map is surjective, and moreover the following short exact sequence splits:
\begin{equation}
 0\to ker \to S\otimes TM \xrightarrow{\delta_\gamma} S \to 0.
\end{equation}
The projection map to the kernel is denoted by $Q\colon S\otimes TM\to S\otimes TM$. More explicitly, in a local oriented orthonormal frame $(e_\al)$ of $M$, a section $\chi\in\Gamma(S\otimes TM)$ can be written as $\chi^\al\otimes e_\al$\footnote{Here and in the sequel, the summation convention is always used.}, and the $Q$-projection is given by
\begin{equation}
 \begin{split}
  Q\chi&\coloneqq -\frac{1}{2}\gamma(e_\be)\gamma(e_\al)\chi^\be\otimes e_\al \\
       &=\frac{1}{2}\left( (\chi^1+\omega\cdot\chi^2)\otimes e_1 -\omega\cdot(\chi^1+\omega\cdot\chi^2)\otimes e_2\right),
 \end{split}
\end{equation}
where $\omega=e_1\cdot e_2$ is the real volume element in the Clifford bundle.

Let $(N,h)$ be a compact Riemannian manifold and $\phi\colon M\to N$ a map. One can consider the twisted spinor bundle $S\otimes \phi^*TN$ with bundle metric $g_s\otimes\phi^*h$ and connection $\widetilde{\na}\equiv\na^{S\otimes\phi^*TN}$,  which is also a Dirac bundle, and the Clifford action on this bundle is also denoted by $\gamma$ or simply a dot. A section of this bundle is called a vector spinor, and it serves as a super partner of the map $\phi$ in this model. The twisted spin Dirac operator $\D$ is defined in the canonical way: let $(e_\al)$ be a local orthonormal frame of $M$, then for any vector spinor $\psi\in \Gamma(S\otimes \phi^*TN)$, define
\begin{equation}
 \D\psi\coloneqq \gamma(e_\al)\widetilde{\na}_{e_\al}\psi=e_\al\cdot \widetilde{\na}_{e_\al}\psi.
\end{equation}
It is elliptic and essentially self-adjoint with respect to the inner product in~$L^2(S\otimes\phi^*TN)$. In a local coordinate $(y^i)$ of $N$, write $\psi=\psi^i\otimes \phi^*(\frac{\p}{\p y^i})$, then
\begin{equation}
 \begin{split}
  \D\psi=\pd\psi^i\otimes \phi^*\left(\frac{\p}{\p y^i}\right)
           +\gamma(e_\al)\psi^i\otimes\phi^*\left(\na^N_{T\phi(e_\al)}\frac{\p}{\p y^i}\right),
 \end{split}
\end{equation}
where $\pd$ is the spin Dirac operator on $S$. For later convention, we set
\begin{equation}
 SR(\psi)\coloneqq\langle\psi^l,\psi^j\rangle_{g_s} \psi^k\otimes \phi^*\left(\Rm^N(\frac{\p}{\p y^k},\frac{\p}{\p y^l})\frac{\p}{\p y^j}\right)=R^i_{\;jkl}(\phi)\langle\psi^l,\psi^j\rangle\psi^k\otimes\phi^*\left(\frac{\p}{\p y^i} \right),
\end{equation}
and $\Rm(\psi)\coloneqq\langle SR(\psi),\psi\rangle_{g_s\otimes\phi^*h}$.

The action functional under consideration is given by
\begin{equation}
 \label{action functional}
 \begin{split}
  \A(\phi, \psi;g, \chi)\coloneqq \int_M
  & |\dd \phi|_{g\otimes \phi^*h}^2
	+ \langle \psi, \D \psi \rangle_{g_s\otimes \phi^*h}
   -4\langle (\mathds{1}\otimes\phi_*)(Q\chi), \psi \rangle_{g_s\otimes\phi^*h}\\
  &   -|Q\chi|^2_{g_s\otimes g} |\psi|^2_{g_s\otimes \phi^*h}
	-\frac{1}{6} \Rm(\psi) \dd vol_g,
 \end{split}
\end{equation}
From \cite{jost2016regularity} we know that the Euler--Lagrange equations are
\begin{equation}\label{EL}
 \begin{split}
 \tau(\phi)=&\frac{1}{2}\Rm(\psi, e_\al\cdot\psi)\phi_* e_\al-\frac{1}{12}S\na R(\psi)   \\
            &  -(\langle \na^s_{e_\be}(e_\al \cdot e_\be \cdot \chi^\al), \psi \rangle_{g_s}
	        	+ \langle e_\al \cdot e_\be \cdot \chi^\al, \widetilde{\na}_{e_\be} \psi \rangle_{g_s}),  \\
    \D\psi =& |Q\chi|^2\psi +\frac{1}{3}SR(\psi)+2(\mathds{1}\otimes \phi_*)Q\chi,
 \end{split}
\end{equation}
where $S\na R(\psi)=\phi^*(\na^N R)_{ijkl}\langle\psi^i,\psi^k\rangle_{g_s} \langle\psi^j,\psi^l\rangle_{g_s}$, and
\begin{equation}
 \begin{split}
  \Rm(\psi, e_\al\cdot\psi)\phi_* e_\al
 =&\langle\psi^k,e_\al\cdot\psi^l \rangle_{g_s} e_\al(\phi^j) \phi^*\left(\Rm\left(\frac{\p}{\p y^k},\frac{\p}{\p y^l}\right) \frac{\p}{\p y^j} \right) \\
 =&R^i_{\;jkl}\langle\psi^k,\na\phi^j\cdot\psi^l\rangle\otimes \phi^*\left(\frac{\p}{\p y^i}\right).
 \end{split}
\end{equation}

One notices that this action functional can actually be defined for $(\phi,\psi)$ that possess only little regularity; we only need integrability properties to make the action well defined, that is, $\phi\in W^{1,2}(M,N)$ and $\psi\in \Gamma^{1,4/3}(S\otimes\phi^*TN)$. The corresponding solutions of \eqref{EL} in the sense of distributions are called weak solutions. When the Riemannian metric~$g$ and the gravitino $\chi$ are assumed to be smooth parameters, it is shown in \cite{jost2016regularity} that any weak solution $(\phi,\psi)$ is actually smooth. We will show that these solutions have more interesting geometric and analytical properties. Embed $(N,h)$ isometrically into some Euclidean space~$\R^K$. Then a solution can be represented by a tuple of functions~$\phi=(\phi^1,\cdots,\phi^K)$ taking values in $\R^K$ and a tuple of spinors $\psi=(\psi^1,\cdots,\psi^K)$ where each $\psi^i$ is a (pure) spinor and they together satisfy the condition that at each point $\phi(x)$ in the image, for any normal vector $\nu=(\nu^1,\cdots,\nu^K)\in T^\perp_{\phi(x)} N\subset T_{\phi(x)}\R^K$,
\begin{equation}
 \sum_{i=1}^K \psi^i(x)\nu^i(\phi(x))=0.
\end{equation}
Moreover, writing the second fundamental form of the isometric embedding as $A=(A^i_{jk})$, the Euler--Lagrange equations can be written in the following form (see \cite{jost2016regularity})
\begin{equation}\label{EL-phi-extrinsic}
 \begin{split}
  \Delta\phi^i
    =&A^i_{jk}\langle\na\phi^j,\na\phi^k\rangle+A^i_{jm}A^m_{kl}\langle\psi^j,\na\phi^k\cdot\psi^l\rangle \\
     &      +Z^i(A,\na A)_{jklm} \langle\psi^j,\psi^l\rangle\langle\psi^k,\psi^m\rangle
          -\diverg V^i-A^i_{jk}\langle V^j,\na\phi^k\rangle,
 \end{split}
\end{equation}
\begin{equation}\label{EL-psi-extrinsic}
 \begin{split}
  \pd\psi^i
    =&-A^i_{jk}\na\phi^j\cdot\psi^k+|Q\chi|^2\psi^i
       +\frac{1}{3}A^i_{jm}A^m_{kl}\left(\langle\psi^k,\psi^l\rangle\psi^j-\langle\psi^j,\psi^k\rangle\psi^l\right)\\
     &  -e_\al\cdot\na\phi^i\cdot\chi^\al.
 \end{split}
\end{equation}
Here the $V^i$'s are vector fields on $M$ defined by
\begin{equation}\label{def of V}
V^i= \langle e_\al\cdot e_\be\cdot\chi^\al,\psi^i\rangle e_\be.
\end{equation}
One should note that there is some ambiguity here, because the second fundamental form maps tangent vectors of the submanifold $N$ to normal vectors, so the lower indices of $A^i_{jk}$ should be tangential indices, and the upper ones normal. However, one can extend the second fundamental form to a tubular neighborhood of $N$ in $\R^K$ such that all the $A^i_{jk}$'s make sense. Alternatively, one can rewrite the extrinsic equations without labeling indices, but we want to derive estimates  and see how the second fundamental form $A$ affects the system, hence we adopt this formulation.

This action functional is closely related to Dirac-harmonic maps with curvature term. Actually, if the gravitinos vanish in the model, the action $\A$ then reads
\begin{equation}
 L_c(\phi,\psi)=\int_M |\dd\phi|^2+\langle\psi,\D\psi\rangle-\frac{1}{6}\Rm(\psi)\dv_g,
\end{equation}
whose critical points are known as Dirac-harmonic maps with curvature term. These were firstly introduced in \cite{chen2008liouville} and further investigated in \cite{branding2015some, branding2016energy, jost2015geometric}. Furthermore, if the curvature term is also omitted, then we get the Dirac-harmonic map functional which was  introduced in \cite{chen2005regularity, chen2006dirac} and further explored from the perspective of geometric analysis in e.g. \cite{zhao2006energy, zhu1, zhu2, wang2009regularity, CJWZ, liu, sharp2016regularity}. From the physical perspective, they constitute a simplified version of the model considered in this paper, and describe the behavior of the nonlinear sigma models in degenerate cases.

The symmetries of this action functional always play an important role in the study of the solution spaces, and here especially the rescaled conformal invariance.
\begin{lemma}\label{conformal transformation lemma}
 Let $f\colon (\widetilde{M},\tilde{g})\to (M,g)$ be a conformal diffeomorphism, with $f^*g=e^{2u}\tilde{g}$, and suppose the spin structure of $(\widetilde{M}, \tilde{g})$ is isomorphic to the pullback of the given one of $(M,g)$. There is an identification $B\colon S\to \tilde{S}$ which is an isomorphism and fiberwise isometry such that under the transformation
 \begin{equation}
  \begin{split}
   \phi &\mapsto \tilde{\phi}\coloneqq\phi\circ f,\\
   \psi &\mapsto \tilde{\psi}\coloneqq e^{\frac{u}{2}}(B\otimes\mathds{1}_{\phi^*TN})\psi, \\
   \chi &\mapsto \tilde{\chi}\coloneqq e^{\frac{3u}{2}}(B\otimes(f^{-1})_*)\chi, \\
    g   &\mapsto \tilde{g},
  \end{split}
 \end{equation}
 each summand of the action functional stays invariant, and also
 \begin{equation}
  \int_M |\psi|^4 \dv_g= \int_{\widetilde{M}} |\tilde{\psi}|^4 \dv_{\tilde{g}}.
 \end{equation}
\end{lemma}
\begin{remark}
 Furthermore, the following quantities are also invariant under the transformations in the above lemma:
 \begin{align}
  \int_{M}|\chi|^4\dd x, & & \int_M |\widetilde{\na}\psi|^{\frac{4}{3}} \dd x, & &\int_M |\widehat{\na}\chi|^{\frac{4}{3}}\dd x,
 \end{align}
 where $\widehat{\na}\equiv \na^{S\otimes TM}$.
 Also observe that $Q$ is only a linear projection operator, so $Q\chi$ enjoys the same analytic properties as $\chi$. In our model, most time it is only the $Q$-part of $\chi$ which is involved, so all the assumptions and conclusions can be made on the $Q\chi$'s.  The rescaled conformal invariance with respect to $\psi$ was shown in \cite{hitchin1974harmonic}, and see also \cite{chen2006dirac}. As for the gravitino $\chi$, the spinor part has to be rescaled in the same way as $\psi$, while the tangent vector part has to be rescaled in the ordinary way, which gives rise to an additional factor $e^u$, such that the corresponding norms are invariant. For more detailed investigations one can refer to \cite{jost2016note} where more symmetry properties of our nonlinear sigma model with gravitinos are analyzed.
\end{remark}

\begin{ex}
 When the map $f$ is a rescaling by a constant $\lambda$ on the Euclidean space with the standard Euclidean metric $g_0$, then $f^*g_0=\lambda^2 g_0$ and $(f^{-1})_*$ is a rescaling by $\lambda^{-1}$. In this case the gravitino $\chi$ transforms to $\sqrt{\lambda}B\chi^\al\otimes e_\al$, where $e_\al$ is a standard basis for $(\R^2,g_0)$.
\end{ex}

For a given pair $(\phi,\psi)$ and a domain $U\subset M$, the energy of this pair $(\phi,\psi)$ on $U$ is defined to be
\begin{equation}
 E(\phi,\psi;U)\coloneqq \int_U |\dd\phi|^2+|\psi|^4\dv_g,
\end{equation}
and when $U$ is the entire manifold we write $E(\phi,\psi)$ omitting $U$. Similarly, the energy of the map $\phi$ resp. the vector spinor $\psi$ on $U$ is defined by
\begin{align}
 E(\phi;U)\coloneqq \int_U |\dd\phi|^2\dv_g, & &\textnormal{resp.} & & E(\psi;U)\coloneqq \int_U |\psi|^4\dv_g.
\end{align}
From the previous lemma we know that they are rescaling invariant. We will show that whenever the local energy of a solution is small, then some higher derivatives of this solution can be controlled by its energy and some appropriate norm of the gravitino; this is known as the small energy regularity. On the other hand, similar to the theories for harmonic maps and Dirac-harmonic maps, the energy of a solution should not be globally small, that is, when the domain is a  closed surface, in particular the standard sphere, because too small energy forces the solution to be trivial. That is, there are energy gaps between the trivial and nontrivial solutions of \eqref{EL} on some closed surfaces.  These will be shown in Section 2.

To proceed further we restrict to some special gravitinos, i.e. \emph{those gravitinos that are critical with respect to variations}. As shown in \cite{jost2016note}, this is equivalent to the vanishing of the corresponding $\it{supercurrent}$. Then we will see in Section 3 that the energy-momentum tensor, defined using a local orthonormal frame $(e_\al)$ by
\begin{equation}
 \begin{split}
  T
  =&\big\{2\langle \phi_*e_\al,\phi_* e_\be\rangle-|\dd\phi|^2 g_{\al\be}
     +\frac{1}{2}\left\langle\psi,e_\al\cdot\widetilde{\na}_{e_\be}\psi+e_\be\cdot\widetilde{\na}_{e_\al}\psi\right\rangle
                        -\langle\psi,\D_g\psi\rangle g_{\al\be}\\
   & +\langle e_\eta\cdot e_\al\cdot\chi^\eta\otimes \phi_*e_\be+e_\eta\cdot e_\be\cdot\chi^\eta\otimes\phi_*e_\al, \psi\rangle
                 +4\langle(\mathds{1}\otimes\phi_*)Q\chi,\psi\rangle g_{\al\be}\\
   &\qquad + |Q\chi|^2 |\psi|^2 g_{\al\be}
                 +\frac{1}{6}\Rm(\psi)g_{\al\be} \big\}e^\al\otimes e^\be,
 \end{split}
\end{equation}
is symmetric, traceless and divergence free, see Proposition \ref{properties of energy-momentum tensor}. Hence it gives rise to a holomorphic quadratic differential, see Proposition \ref{holomorphicity of T(z)}. In a local conformal coordinate $z=x+iy$, this differential reads
\begin{equation}
 T(z)\dd z^2\coloneqq(T_{11}-iT_{12})(\dd x+i\dd y)^2,
\end{equation}
with
\begin{equation}
 \begin{split}
  T_{11}
  &=\left|\frac{\p\phi}{\p x}\right|^2-\left|\frac{\p\phi}{\p y}\right|^2
   +\frac{1}{2}\left(\langle\psi,\gamma(\p_x)\widetilde{\na}_{\p_x}\psi\rangle-\langle\psi,\gamma(\p_y)\widetilde{\na}_{\p_y}\psi\rangle\right)
    +F_{11},\\
  T_{12}
  &=\left\langle\frac{\p\phi}{\p x},\frac{\p\phi}{\p y}\right\rangle_{\phi^*h}
      +\langle\psi,\gamma(\p_x)\widetilde{\na}_{\p_y}\psi\rangle+F_{12},
 \end{split}
\end{equation}
where in a local chart $\chi=\chi^x\otimes \p_x+\chi^y\otimes\p_y$ and
\begin{equation}
 \begin{split}
  F_{11}&=2\langle-\chi^x\otimes\phi_*(\p_x)-\gamma(\p_x)\gamma(\p_y)\chi^y\otimes\phi_*(\p_x),\psi\rangle
           +2\langle(\mathds{1}\otimes\phi_*)Q\chi, \psi\rangle g(\p_x,\p_x), \\
  F_{12}&=2\langle-\chi^x\otimes\phi_*(\p_y)-\gamma(\p_x)\gamma(\p_y)\chi^y\otimes \phi_*(\p_y),\psi\rangle.
 \end{split}
\end{equation}
Consequently we can establish a Pohozaev type identity for our model in Section 4. This will be the key ingredient for the analysis in the sequel.
\begin{thm}[restate=Pohozaev, label=Pohozaev Identity]
 \textnormal{\textbf{(Pohozaev identity.)}}
 Let $(\phi,\psi)$ be a smooth solution of \eqref{EL} on $B_1^*:=B_1\backslash\{0\}$ with $\chi$ being a critical gravitino which is smooth on $B_1$. Assume that $(\phi,\psi)$ has finite energy on $B_1$. Then for any $0<r<1$,
 \begin{equation}\label{Pohozaev identity}
  \begin{split}
   \int_0^{2\pi}\left|\frac{\p\phi}{\p r}\right|^2-\frac{1}{r^2}\left|\frac{\p\phi}{\p\theta}\right|^2 \dd\theta
   &=\int_0^{2\pi} -\langle\psi, \gamma(\p_r)\widetilde{\na}_{\p_r}\psi\rangle+\frac{1}{6}\Rm(\psi) -(F_{11}\cos{2\theta}+F_{12}\sin{2\theta})\dd\theta \\
   =&\int_0^{2\pi} \left\langle\psi,\frac{1}{r^2}\gamma(\p_\theta)\widetilde{\na}_{\p_\theta}\psi\right\rangle-\frac{1}{6}\Rm(\psi)- (F_{11}\cos{2\theta}+F_{12}\sin{2\theta})\dd\theta.
  \end{split}
 \end{equation}
\end{thm}

In Section 4 we also prove that isolated singularities are removable, using a result from the Appendix and the regularity theorem in \cite{jost2016regularity}.

Finally, for a sequence of solutions $(\phi_k,\psi_k)$ with uniformly bounded energies defined on $(M,g)$ with respect to critical gravitinos $\chi_k$ which converge in $W^{1,\frac{4}{3}}$ to some smooth limit $\chi$, a subsequence can be extracted which converges weakly in $W^{1,2}\times L^4$ to a solution defined on $(M,g)$, and by a rescaling argument, known as the blow-up procedure, we can get some solutions with vanishing gravitinos, i.e. Dirac harmonic maps with curvature term, defined on the standard sphere $\sph^2$ with target manifold $(N,h)$, known as ``bubbles''. Moreover, the energies pass to the limit, i.e. the energy identities hold.

\begin{thm}\label{energy identity thm}
 \textnormal{\textbf{(Energy identities.)}}
 Let $(\phi_k,\psi_k)$ be a sequence of solutions of \eqref{EL} with respect to smooth critical gravitinos~$\chi_k$ which converge in $W^{1,\frac{4}{3}}$ to a smooth limit $\chi$, and assume their energies are uniformly bounded:
 \begin{equation}
  E(\phi_k,\psi_k)\le \Lambda<\infty.
 \end{equation}
 Then passing to a subsequence if necessary, the sequence $(\phi_k,\psi_k)$ converges weakly in the space~$W^{1,2}(M,N)\times L^4(S\otimes \R^K)$ to a smooth solution~$(\phi,\psi)$ with respect to $\chi$. Moreover, the blow-up set
 \begin{equation}
  \mathcal{S}\coloneqq \bigcap_{r>0}\left\{p\in M\Big| \liminf_{k\to \infty}\int_{B_{r}(p)}|\na\phi_k|^2+|\psi_k|^4 \dv_g\ge\vep_0\right\}
 \end{equation}
 is a finite (possibly empty) set of points $\{p_1,\dots,p_I\}$, and correspondingly a finite set (possibly empty) of Dirac-harmonic maps with curvature term $(\sigma^l_i,\xi^l_i)$ defined on $\sph^2$ with target manifold~$(N,h)$, for $l=1,\dots,L_i$ and $i=1,\dots, I$, such that the following energy identities hold:
 \begin{equation}
  \lim_{k\to\infty} E(\phi_k)=E(\phi)+\sum_{i=1}^{I} \sum_{l=1}^{L_i} E(\sigma_i^l),
 \end{equation}
 \begin{equation}
  \lim_{k\to\infty} E(\psi_k)=E(\psi)+\sum_{i=1}^{I} \sum_{l=1}^{L_i} E(\xi_i^l).
 \end{equation}
\end{thm}

The proof will be given in Section 5. Although these conclusions are similar to those for harmonic maps and Dirac-harmonic maps and some of its variants in e.g. \cite{jost1991two, parker1996bubble, chen2005regularity, zhao2006energy, jost2015geometric}, one has to pay special attentions to the critical gravitinos.

%%%%%%%%%%%%%%%%%%%%%%%%%%%%%%%%----2----%%%%%%%%%%%%%%%%%%%%%%%%%%%%%%%%%%%%%%%%%%%%%%%%%%%%%%%%%%%%%%%%%%%%%%%%%%%%%
\section{Small energy regularity and energy gap property}

In this section we consider the behavior of solutions with small energies.

%%%%%%%%%%%%%%%%%---2.1---%%%%%%%%%%%%%%%%%%%%%%%%%%%%%%%%
\subsection{}

First we show the small energy regularities. Recall that for harmonic maps and Dirac-harmonic maps and its variants \cite{sacks1981existence, chen2006dirac,  branding2016energy, jost2015geometric}, it suffices to assume that the energy on a local domain is small. However, as we will see soon, here we have to assume that the gravitinos are also small. For the elliptic estimates used here, one can refer to \cite{begehr1994complex, gilbarg2001elliptic, chen2008nonlinear}, or more adapted versions in \cite{ammann2003variational}.

\begin{thm}\label{thm-small energy regularity}
 \textnormal{\textbf{($\vep_1$-Regularity theorem.)}}
 Consider the local model defined on the Euclidean unit disk $B_1\subset \R^2$ and the target manifold is a submanifold $(N,h)\hookrightarrow\R^K$ with second fundamental form $A$.  For any $p_1\in (1,\frac{4}{3})$ and $p_2\in(1,2)$ there exists an $\vep_1=\vep_1(A,p_1,p_2)\in (0,1)$ such that if the gravitino $\chi$ and a solution $(\phi,\psi)$ of \eqref{EL} satisfy
 \begin{align}
  E(\phi,\psi;B_1)=\int_{B_1}|\na\phi|^2+|\psi|^4\dd x\le \vep_1,
  & & \int_{B_1}|\chi|^4+|\widehat{\na}\chi|^{\frac{4}{3}}\dd x\le \vep_1,
 \end{align}
 then for any $U\Subset B_1$, the following estimates hold:
 \begin{equation}
  \|\phi\|_{W^{2,p_1}(U)}
  \le C\big(|A|\|\na\phi\|_{L^2(B_1)}+|A|^2\|\psi\|^2_{L^4(B_1)}+\|Q\chi\|_{W^{1,\frac{4}{3}}(B_1)}\big),
 \end{equation}
 \begin{equation}
  \|\psi\|_{W^{1,p_2}(U)}
  \le C\big(|A|\|\na\phi\|_{L^2(B_1)}+|A|^2\|\psi\|^2_{L^4(B_1)}+\|Q\chi\|_{W^{1,\frac{4}{3}}(B_1)}\big),
 \end{equation}
where $C=C(p_1, p_2,U,N)>0$.
\end{thm}

\begin{remark}
 Note that if the second fundamental form $A$ vanishes identically, then $N$ is a totally geodesic submanifold of the Euclidean space $\R^K$, hence there is no curvatures on $N$ and the model then is  easy and not of interest. So we will assume that $A\neq 0$, and without loss of generality, we assume $|A|\equiv \|A\|\ge 1$. For some $C(p)$ depending on the value of $p$ to be chosen later, the small barrier constant $\vep_1$ will be required to satisfy
 \begin{align}\label{constraints for vep_1}
  C(p)|A|^2\sqrt{\vep_1}\le \frac{1}{8}, & & C(p)|A| |\na A| \vep^{3/4}\le\frac{1}{8},
 \end{align}
 where $|\na A|\equiv\|\na A\|$. These restrictions will be explained in the proof.
\end{remark}
\begin{remark}
 Note also that since the domain is the Euclidean disk $B_1$, the connection $\widehat{\na}$ is actually equivalent to $\na^s$.
\end{remark}

\begin{proof}[Proof of Theorem \ref{thm-small energy regularity}]
Since $N$ is taken as a compact submanifold of $\R^K$, we may assume that it is contained in a ball of radius $C_N$ in $\R^K$, which implies $\|\phi\|_{L^\infty}\le C_N$. Moreover, as we are dealing with a local solution $(\phi,\psi)$, we may assume that $\int_{B_1}\phi\dd x=0$, so that the Poincar\'e inequalities hold: for any $p\in [1,\infty]$,
\begin{equation}
 \|\phi\|_{L^p(B_1)}\le C_p\|\na\phi\|_{L^p(B_1)}.
\end{equation}
Let $(U_k)_{k\ge 1}$ be a sequence of nonempty disks such that
\begin{equation}
 B_1\Supset U_1\Supset U_2\Supset U_3\Supset \cdots.
\end{equation}

Take a smooth cutoff function $\eta\colon B_1\to \R$ such that $0\le \eta\le 1$, $\eta|_{U_1}\equiv 1$, and $\supp\eta\subset B_1$. Then $\eta\psi$ satisfies
 \begin{equation}\label{cutoff equation for psi}
  \begin{split}
   \pd(\eta\psi^i)
   =&\na\eta\cdot\psi^i+\eta\pd\psi^i \\
   =&\na\eta\cdot\psi^i-A^i_{jk}\na\phi^j\cdot(\eta\psi^k)+|Q\chi|^2(\eta\psi^i) \\
    &+\frac{1}{3}A^i_{jm}A^m_{kl}\left(\langle\psi^k,\psi^l\rangle(\eta\psi^j)
                                            -\langle\psi^j,\psi^k\rangle(\eta\psi^l)\right) \\
    &   -e_\al\cdot\na(\eta\phi^i)\cdot\chi^\al+e_\al\cdot\phi^i\na\eta\cdot\chi^\al.
  \end{split}
 \end{equation}
Then one has
\begin{equation}\label{estimating psi}
 \begin{split}
 |\pd(\eta\psi)|\le
   &|\na\eta||\psi|+|A||\na\phi|\cdot|\eta\psi|+|Q\chi|^2|\eta\psi|+|A|^2|\psi|^2\cdot|\eta\psi| \\
   &+|Q\chi||\na(\eta\phi)|+|\phi||\na\eta||Q\chi|.
 \end{split}
\end{equation}
 Consider the $L^{p}$-norm (where $p\in (1,2)$) of the left hand side:
 \begin{equation}
  \begin{split}
   \|\pd(\eta\psi)\|_{L^{p}(B_1)}
   \le& \|\na\eta\|_{L^{\frac{4p}{4-p}}(B_1)}\|\psi\|_{L^4(B_1)}
          +|A|\|\na\phi\|_{L^2(B_1)}\|\eta\psi\|_{L^\frac{2p}{2-p}(B_1)}\\
      &    +\|Q\chi\|^2_{L^4(B_1)}\|\eta\psi\|_{L^{\frac{2p}{2-p}}(B_1)}
        +|A|^2\|\psi\|^2_{L^4(B_1)} \|\eta\psi\|_{L^{\frac{2p}{2-p}}(B_1)}\\
      &   +\|Q\chi\|_{L^4(B_1)}\|\na(\eta\phi)\|_{L^{\frac{4p}{4-p}}(B_1)}
         +C_N\|\na\eta\|_{L^{\frac{4p}{4-p}}(B_1)}\|Q\chi\|_{L^4(B_1)}.
  \end{split}
 \end{equation}
 Assume that $\|\na\eta\|_{L^{\frac{4p}{4-p}}(B_1)}$ is bounded by some constant $C'=C'(U_1,p)$.
 Since $\eta\psi$ vanishes on the boundary and $\pd$ is an elliptic operator of order one, we have
 \begin{equation}
  \|\eta\psi\|_{L^{\frac{2p}{2-p}}(B_1)}\le C(p)\|\na^s(\eta\psi)\|_{L^p(B_1)}
  \le C(p)\|\pd(\eta\psi)\|_{L^p(B_1)}.
 \end{equation}
 Then from
 \begin{equation}
  \begin{split}
   \|\pd(\eta\psi)\|_{L^p(B_1)}
  \le& \left(|A|\|\na\phi\|_{L^2(B_1)}+\|Q\chi\|^2_{L^4(B_1)}+|A|^2\|\psi\|^2_{L^4(B_1)}\right)
                      \|\eta\psi\|_{L^{\frac{2p}{2-p}}(B_1)} \\
     & +\|Q\chi\|_{L^4(B_1)}\|\na(\eta\phi)\|_{L^{\frac{4p}{4-p}}(B_1)}
       +C_N C'\left(\|\psi\|_{L^4(B_1)} +\|Q\chi\|_{L^4(B_1)}\right)
  \end{split}
 \end{equation}
 together with the fact
 \begin{equation}
  |\widetilde{\na}(\eta\psi)|\le |\na^s(\eta\psi)|+|A||\eta\psi||\na\phi|,
 \end{equation}
 it follows that
 \begin{equation}\label{first estimate for psi}
  \begin{split}
     \|\widetilde{\na}(\eta\psi)\|_{L^p(B_1)}
  \le 2C(p)\left(\|Q\chi\|_{L^4(B_1)}\|\na(\eta\phi)\|_{L^{\frac{4p}{4-p}}(B_1)}
       +C_N C'\left(\|\psi\|_{L^4(B_1)} +\|Q\chi\|_{L^4(B_1)}\right)\right)
  \end{split}
 \end{equation}
 provided that \eqref{constraints for vep_1} is satisfied.

 %%%%%%%%%%%%%%%%%%
 Now consider the map $\phi$. The equations for $\eta\phi$ are
 \begin{equation}
  \begin{split}
   \Delta(\eta\phi^i)
   =&\eta\Delta\phi^i+2\langle\na\eta,\na\phi^i\rangle+(\Delta\eta)\phi^i \\
   =&\eta\big(A^i_{jk}\langle\na\phi^j,\na\phi^k\rangle+A^i_{jm}A^m_{kl}\langle\psi^j,\na\phi^k\cdot\psi^l\rangle \\
    &    +Z^i(A,\na A)_{jklm} \langle\psi^j,\psi^l\rangle\langle\psi^k,\psi^m\rangle \\
    & -\diverg V^i-A^i_{jk}\langle V^j,\na\phi^k\rangle\big)+2\langle\na\eta,\na\phi^i\rangle+(\Delta\eta)\phi^i.
  \end{split}
 \end{equation}
 Using $\eta\na\phi^i=\na(\eta\phi^i)-\phi^i(\na\eta)$, we can rewrite it as
 \begin{equation}\label{cutoff equation for phi}
  \begin{split}
   \Delta(\eta\phi^i)
   =&A^i_{jk}\langle\na\phi^j,\na(\eta\phi^k)\rangle+A^i_{jm}A^m_{kl}\langle\psi^j,\na(\eta\phi^k)\cdot\psi^l\rangle \\
    &  +Z^i(A,\na A)_{jklm} \langle\psi^j,\psi^l\rangle\langle\psi^k,\eta\psi^m\rangle \\
    &-\diverg(\eta V^i)-A^i_{jk}\langle V^j,\na(\eta\phi^k)\rangle+2\langle\na\eta,\na\phi^i\rangle+(\Delta\eta)\phi^i \\
    &-A^i_{jk}\langle\na\phi^j,\phi^k\na\eta\rangle-A^i_{jm}A^m_{kl}\langle\psi^j,\phi^k\na\eta\cdot\psi^l\rangle
      +\langle\na\eta,V^i\rangle
      +A^i_{jk}\langle V^j,\phi^k\na\eta\rangle.
  \end{split}
 \end{equation}
 Notice that $\eta\phi^i\in C^\infty_0(B_1)$. Split it as $\eta\phi^i=u^i+v^i$, where $u^i\in C^\infty_0(B_1)$ uniquely solves (see e.g. \cite[Chap. 8]{chen1998elliptic})
 \begin{equation}
  \Delta u^i=-\diverg(\eta V^i).
 \end{equation}
 Since $\eta V^i\in L^{\frac{4p}{4-p}}(B_1)$, it follows from the $L^p$ theory of Laplacian operators that
 \begin{equation}\label{first estimate for u}
  \|u\|_{W^{1,\frac{4p}{4-p}}(B_1)}\le C(p)\|Q\chi\|_{L^4(B_1)}\|\eta\psi\|_{L^{\frac{2p}{2-p}}(B_1)}.
 \end{equation}
 Then $v^i\in C^\infty_0(B_1)$ satisfies
 \begin{equation}
  \begin{split}
   \Delta v^i
   =&\Delta(\eta\phi^i)-\Delta u^i \\
   =&A^i_{jk}\langle\na\phi^j,\na(\eta\phi^k)\rangle+A^i_{jm}A^m_{kl}\langle\psi^j,\na(\eta\phi^k)\cdot\psi^l\rangle \\
    &  +Z^i(A,\na A)_{jklm} \langle\psi^j,\psi^l\rangle\langle\psi^k,\eta\psi^m\rangle
       -A^i_{jk}\langle V^j,\na(\eta\phi^k)\rangle+2\langle\na\eta,\na\phi^i\rangle+(\Delta\eta)\phi^i \\
    &-A^i_{jk}\langle\na\phi^j,\phi^k\na\eta\rangle-A^i_{jm}A^m_{kl}\langle\psi^j,\phi^k\na\eta\cdot\psi^l\rangle
      +\langle\na\eta,V^i\rangle +A^i_{jk}\langle V^j,\phi^k\na\eta\rangle.
  \end{split}
 \end{equation}
 From \cite{jost2016regularity}, $\|Z(A,\na A)\|\le |A||\na A|$. Thus the $L^{\frac{4p}{4+p}}$ norm of $\Delta v$ can thus be estimated by
 \begin{equation}
  \begin{split}
   \|\Delta v\|_{L^{\frac{4p}{4+p}}(B_1)}
   &\le |A|\|\na\phi\|_{L^2(B_1)}\|\na(\eta\phi)\|_{L^{\frac{4p}{4-p}}(B_1)}
        +|A|^2\|\psi\|^2_{L^4(B_1)}\|\na(\eta\phi)\|_{L^{\frac{4p}{4-p}}(B_1)} \\
   +&|A||\na A|\|\psi\|^3_{L^4(B_1)}\|\eta\psi\|_{L^{\frac{2p}{2-p}}(B_1)}
     +|A|\|Q\chi\|_{L^4(B_1)}\|\psi\|_{L^4(B_1)}\|\na(\eta\phi)\|_{L^{\frac{4p}{4-p}}(B_1)} \\
   +&2\|\na\eta\|_{L^{\frac{4p}{4-p}}(B_1)}\|\na\phi\|_{L^2(B_1)}
     +\|\Delta\eta\|_{L^{\frac{4p}{4-p}}(B_1)}\|\phi\|_{L^2(B_1)}\\
   +&|A|\|\na\phi\|_{L^2(B_1)}\|\phi\na\eta\|_{L^{\frac{4p}{4-p}}(B_1)}
      +|A|^2\|\psi\|^2_{L^4(B_1)}\|\phi\na\eta\|_{L^{\frac{4p}{4-p}}(B_1)} \\
   +&\|\na\eta\|_{L^{\frac{4p}{4-p}}(B_1)}\|Q\chi\|_{L^4(B_1)}\|\psi\|_{L^4(B_1)}
     +|A|\|\phi\na\eta\|_{L^{\frac{4p}{4-p}}(B_1)}\|Q\chi\|_{L^4(B_1)}\|\psi\|_{L^4(B_1)}.
  \end{split}
 \end{equation}
 As before assume that $\|\na\eta\|_{L^{\frac{4p}{4-p}}(B_1)}$ and $\|\Delta\eta\|_{L^{\frac{4p}{4-p}}(B_1)}$ are bounded by $C'=C'(U_1,p)$. Collecting the terms, we get
 \begin{equation}
  \begin{split}
   \|\Delta v\|_{L^{\frac{4p}{4+p}}(B_1)}
  \le& \big(|A|\|\na\phi\|_{L^2(B_1)}+|A|^2\|\psi\|^2_{L^2(B_1)}+|A|\|Q\chi\|_{L^4(B_1)}\|\psi\|_{L^4(B_1)}\big)
                  \|\na(\eta\phi)\|_{L^{\frac{4p}{4-p}}(B_1)} \\
   & +|A||\na A|\|\psi\|^3_{L^4(B_1)}\|\eta\psi\|_{L^{\frac{2p}{2-p}}(B_1)}  \\
   & +C' C_N \big(2\|\na\phi\|_{L^2(B_1)}+\|\phi\|_{L^2(B_1)}+|A|\|\na\phi\|_{L^2(B_1)}+|A|^2\|\psi\|^2_{L^4(B_1)} \\
   &\qquad\qquad\qquad +\|Q\chi\|_{L^4(B_1)}\|\psi\|_{L^4(B_1)}+|A|\|Q\chi\|_{L^4(B_1)}\|\psi\|_{L^4(B_1)}\big).
  \end{split}
 \end{equation}
 By Sobolev embedding,
 \begin{equation}\label{first estimate for v}
  \begin{split}
   \|v\|_{W^{1,\frac{4p}{4-p}}(B_1)}
   &\le C(p)\big(|A|\|\na\phi\|_{L^2(B_1)}+|A|^2\|\psi\|^2_{L^4(B_1)}+\|Q\chi\|^2_{L^4(B_1)}\big)
                 \|\na(\eta\phi)\|_{L^{\frac{4p}{4-p}}(B_1)} \\
     & +C(p)|A||\na A|\|\psi\|^3_{L^4(B_1)}\|\eta\psi\|_{L^{\frac{2p}{2-p}}(B_1)}  \\
     & + 4 C(p) C' C_N\big(|A|\|\na\phi\|_{L^2(B_1)}+|A|^2\|\psi\|^2_{L^4(B_1)}+\|Q\chi\|^2_{L^4(B_1)}\big).
  \end{split}
 \end{equation}
 Since $\eta\phi=u+v$, combining \eqref{first estimate for u} and \eqref{first estimate for v}, we get
 \begin{equation}\label{first estimate for phi}
  \begin{split}
   \|\eta\phi\|_{W^{1,\frac{4p}{4-p}}(B_1)}
  &\le C(p)\big(|A|\|\na\phi\|_{L^2(B_1)}+|A|^2\|\psi\|^2_{L^4(B_1)}+\|Q\chi\|^2_{L^4(B_1)}\big)
                 \|\na(\eta\phi)\|_{L^{\frac{4p}{4-p}}(B_1)} \\
     +&C(p)|A||\na A|\|\psi\|^3_{L^4(B_1)}\|\eta\psi\|_{L^{\frac{2p}{2-p}}(B_1)}
         +C(p)\|Q\chi\|_{L^4(B_1)}\|\eta\psi\|_{L^{\frac{2p}{2-p}}(B_1)}\\
     +&4 C(p) C' C_N\big(|A|\|\na\phi\|_{L^2(B_1)}+|A|^2\|\psi\|^2_{L^4(B_1)}+\|Q\chi\|^2_{L^4(B_1)}\big).
  \end{split}
 \end{equation}
 By the small energy assumption, and Sobolev embedding, this implies that
 \begin{equation}
  \begin{split}
   \|\eta\phi\|_{W^{1,\frac{4p}{4-p}}(B_1)}
  \le &2C(p)\big(|A||\na A|\|\psi\|^3_{L^4(B_1)}+\|Q\chi\|_{L^4(B_1)}\big)\|\eta\psi\|_{W^{1,p}(B_1)}\\
     & + 8 C(p) C' C_N\big(|A|\|\na\phi\|_{L^2(B_1)}+|A|^2\|\psi\|^2_{L^4(B_1)}+\|Q\chi\|^2_{L^4(B_1)}\big).
  \end{split}
 \end{equation}
 The estimates \eqref{first estimate for psi} and \eqref{first estimate for phi}, together with the small energy assumption, imply that for any $p\in (1,2)$,
 \begin{equation}\label{first improvement}
  \|\eta\phi\|_{W^{1,\frac{4p}{4-p}}(B_1)}+\|\eta\psi\|_{W^{1,p}(B_1)}
  \le C(p,\eta,N)\big(|A|\|\na\phi\|_{L^2(B_1)}+|A|^2\|\psi\|^2_{L^4(B_1)}+\|Q\chi\|^2_{L^4(B_1)}\big).
 \end{equation}
 Note that as $p\nearrow 2$, $\frac{4p}{4-p}\nearrow 4$. Thus, $\eta\phi$ is almost a $W^{1,4}$ map and $\eta\psi$ is almost a $W^{1,2}$ vector spinor.

 Now $\chi\in W^{1,\frac{4}{3}}$, thus in the equations for the map $\phi$, the divergence terms can be reconsidered. Take another cutoff function, still denoted by $\eta$, such that $0\le \eta\le 1$, $\eta|_{U_2}\equiv 1$, and $\supp\eta\subset U_1$. Then $\eta\phi$ satisfies equations of the same form as \eqref{cutoff equation for phi}, and $\diverg(\eta V^i)\in L^p(B_1)$ for any $p\in[1,\frac{4}{3})$. For example, we take $p=\frac{8}{7}$, then
 \begin{equation}
  \|\diverg(\eta V^i)\|_{L^{\frac{8}{7}}(B_1)}
   \le C(\eta)\|\psi\|_{W^{1,\frac{8}{5}}(U_1)}\|Q\chi\|_{W^{1,\frac{4}{3}}(B_1)},
 \end{equation}
 and note that $\|\psi\|_{W^{1,\frac{8}{5}}(U_1)}$ is under control by \eqref{first improvement}. Recalling \eqref{cutoff equation for phi} we have the estimate
 \begin{equation}
  \begin{split}
   \|\Delta(\eta\phi)\|_{L^{\frac{8}{7}}(B_1)}
   &\le |A|\|\na\phi\|_{L^2(B_1)}\|\na(\eta\phi)\|_{L^{\frac{8}{3}}(B_1)}
      +|A|^2\|\psi\|^2_{L^4(B_1)}\|\na(\eta\phi)\|_{L^{\frac{8}{3}}(B_1)} \\
  +&|A||\na A|\|\psi\|^3_{L^4(B_1)}\|\eta\psi\|_{L^8(B_1)}
      +|A|\|Q\chi\|_{L^4(B_1)}\|\psi\|_{L^4(B_1)}\|\na(\eta\phi)\|_{L^{\frac{8}{3}}(B_1)} \\
  +&\|\diverg(\eta V)\|_{L^{\frac{8}{7}}(B_1)}+2\|\na\eta\|_{L^{\frac{8}{3}}(B_1)}\|\na\phi\|_{L^2(B_1)}
     +\|\Delta\eta\|_{L^{\frac{8}{3}}(B_1)}\|\phi\|_{L^2(B_1)} \\
  +&|A|\|\na\phi\|_{L^2(B_1)}\|\phi\na\eta\|_{L^{\frac{8}{3}}(B_1)}
      +|A|^2\|\psi\|^2_{L^4(B_1)}\|\phi\na\eta\|_{L^{\frac{8}{3}}(B_1)} \\
  +&\|\na\eta\|_{L^{\frac{8}{3}}(B_1)}\|Q\chi\|_{L^4(B_1)}\|\psi\|_{L^4(B_1)}
      +|A|\|\phi\na\eta\|_{L^{\frac{8}{3}}(B_1)}\|Q\chi\|_{L^4(B_1)}\|\psi\|_{L^4(B_1)}.
  \end{split}
 \end{equation}
 As before we assume $\|\na\eta\|_{L^{\frac{8}{3}}(B_1)}$ and $\|\Delta\eta\|_{L^{\frac{8}{3}}(B_1)}$ are bounded by $C''=C''(U_2,U_1)$. Then
 \begin{equation}
  \begin{split}
   \|\Delta(\eta\phi)\|_{L^{\frac{8}{7}}(B_1)}
   \le& \left(|A|\|\na\phi\|_{L^2(B_1)}+|A|^2\|\psi\|^2_{L^4(B_1)}+\|Q\chi\|^2_{L^4(B_1)}\right)
              \|\na(\eta\phi)\|_{L^{\frac{8}{3}}(B_1)}  \\
   &+|A||\na A|\|\psi\|^3_{L^4(B_1)}\|\eta\psi\|_{L^8(B_1)}
     +C(\eta)\|\psi\|_{W^{1,\frac{8}{5}}(U_1)}\|Q\chi\|_{W^{1,\frac{4}{3}}(B_1)} \\
   &+4 C_N C''\left(|A|\|\na\phi\|_{L^2(B_1)}+|A|^2\|\psi\|^2_{L^4(B_1)}+\|Q\chi\|^2_{L^4(B_1)}\right).
  \end{split}
 \end{equation}
 By the smallness assumptions and the $L^p$ theory for Laplacian operator (here $p=\frac{8}{7}$) we get
 \begin{equation}
  \begin{split}
   \|\eta\phi\|_{W^{2,\frac{8}{7}}(B_1)}
   \le & C(p,U_2,N)\left(|A|\|\na\phi\|_{L^2(B_1)}+|A|^2\|\psi\|^2_{L^4(B_1)}+\|Q\chi\|^2_{W^{1,\frac{4}{3}}(B_1)}\right).
  \end{split}
 \end{equation}
  One can check that similar estimates hold for $\|\eta\phi\|_{W^{1,p}(B_1)}$ for any $p\in(1,\frac{4}{3})$. This accomplishes the proof.

\end{proof}

Recall the Sobolev embeddings
\begin{equation}
 W^{2,\frac{8}{7}}_0(B_1)\hookrightarrow W^{1,\frac{8}{3}}_0(B_1)\hookrightarrow C^{1/4}_0(B_1).
\end{equation}
Thus we see that the map $\phi$ is H\"older continuous with
\begin{equation}
 \|\eta\phi\|_{C^{1/4}(B_1)}\le C\|\eta\phi\|_{W^{2,\frac{8}{7}}(B_1)}.
\end{equation}
In particular, when the energies of $(\phi,\psi)$ and certain norms of the gravitino are small, say smaller than $\vep$ (where $\vep\le \vep_1$), the $\frac{1}{4}$-H\"older norm of the map in the interior is also small, with the estimate
\begin{equation}\label{Holder estimate}
 \|\phi\|_{C^{1/4}(U)}\le C(N,U, |A|)\sqrt{\vep}.
\end{equation}

%%%%%%%%%%%%%%%%%%%%%%%%%%%%%%%%%%%%%%%%%%%%%%%%%%%%%%%%%%%%%%%%%%%%%%%%%

\subsection{}
In this subsection we show the existence of energy gaps. For harmonic maps, this is a well known property. On certain closed surfaces the energy gaps are known to exist for Dirac-harmonic maps (with or without curvature term), and using a similar method here we get the following version with gravitinos, compare with \cite[Theorem 3.1]{chen2005regularity}, \cite[Lemma 4.1]{chen2008nonlinear}, \cite[Lemma 4.8, Lemma 4.9]{branding2015some}  and \cite[Proposition 5.2]{ jost2015geometric}.

\begin{prop}[restate=Gap, label=energy gaps]
 \textnormal{\textbf{(Energy gap property.)}}
 Suppose that $(\phi,\psi)$ is a solution to \eqref{EL} defined on an oriented closed surface $(M,g)$ with target manifold $(N,h)$. Suppose that the spinor bundle $S\to (M,g)$ doesn't admit any nontrivial harmonic spinors. Then there exists an $\vep_0=\vep_0(M,g,A)\in (0,1)$ such that if
 \begin{equation}\label{ep-0 assumption}
  E(\phi,\psi)+\|Q\chi\|_{W^{1,\frac{4}{3}}(M)}\le \vep_0,
 \end{equation}
 then $(\phi,\psi)$ has to be a trivial solution.
\end{prop}

The existence of harmonic spinors is closely related to the topological and Riemannian structures. Examples of closed surfaces which don't admit harmonic spinors include $\sph^2$ with arbitrary Riemannian metric and the torus $\mathbb{T}^2$ with a nontrivial spin structure, and many others. For more information on harmonic spinors one can refer to \cite{hitchin1974harmonic, bar1998harmonic}.

\begin{proof}[Proof of Proposition \ref{energy gaps}]
 When the spinor bundle $S$ doesn't admit nontrivial harmonic spinors, the Dirac operator is ``invertible'', in the sense that for any $1<p<\infty$, there holds
 \begin{equation}
  \|\sigma\|_{L^p(M)}\le C(p)\|\pd\sigma\|_{L^p(M)}, \qquad \forall \sigma \in \Gamma(S).
 \end{equation}
 See e.g. \cite{chen2008nonlinear} for a proof\footnote{There they show a proof for $p=\frac{4}{3}$, but it is easily generalized to a general $p\in(1,\infty)$.}. As $\pd$ is an elliptic operator of first order, one has
 \begin{equation}
  \|\na^s\psi^a\|_{L^{\frac{8}{5}}(M)}\le C\left(\|\pd\psi^a\|_{L^{\frac{8}{5}}(M)}+\|\psi^a\|_{L^{\frac{8}{5}}(M)}\right),\qquad 1\le a\le K.
 \end{equation}
 It follows that
 \begin{equation}\label{ellipticity of pd}
  \|\psi\|_{W^{1,\frac{8}{5}}(M)}\le C\|\pd\psi\|_{L^{\frac{8}{5}}(M)}+|A|\|\psi\|_{L^4(M)}\|\na\phi\|_{L^{\frac{8}{3}}(M)}.
 \end{equation}
 From \eqref{EL-psi-extrinsic} one gets
 \begin{equation}
  \begin{split}
   \|\pd\psi\|_{L^{\frac{8}{5}}(M)}
   \le& |A|\|\na\phi\|_{L^2(M)}\|\psi\|_{L^8(M)} +\|Q\chi\|^2_{L^4(M)}\|\psi\|_{L^8(M)} \\
      & \quad +|A|^2\|\psi\|^2_{L^4(M)}\|\psi\|_{L^8(M)}+\|Q\chi\|_{L^4(M)}\|\na\phi\|_{L^{\frac{8}{3}}(M)}.
  \end{split}
 \end{equation}
 Since \eqref{ep-0 assumption} holds, using \eqref{ellipticity of pd} one obtains
 \begin{equation}\label{control psi in terms of phi}
  \|\psi\|_{W^{1,\frac{8}{5}}(M)}\le C\left(\|Q\chi\|_{L^4(M)}+\|\psi\|_{L^4(M)}\right)\|\na\phi\|_{L^{\frac{8}{3}}(M)}.
 \end{equation}

 Next we deal with the map $\phi$. From \eqref{EL-phi-extrinsic} it follows that
 \begin{equation}
  \begin{split}
   \|\Delta\phi\|_{L^{\frac{8}{7}}(M)}
  \le& |A|\|\na\phi\|_{L^2(M)}\|\na\phi\|_{L^\frac{8}{3}(M)}
   +|A|^2\|\psi\|^2_{L^4(M)}\|\na\phi\|_{L^{\frac{8}{3}}(M)} \\
   &+|A||\na A|\|\psi\|^3_{L^4(M)}\|\psi\|_{L^8(M)}
    +\left(\|\widehat{\na}Q\chi\|_{L^{\frac{4}{3}}(M)}+C\|Q\chi\|_{L^{\frac{4}{3}}(M)}\right)\|\psi\|_{L^8(M)} \\
   &+\|Q\chi\|_{L^4(M)}\|\widetilde{\na}\psi\|_{L^{\frac{8}{5}}(M)}
     +|A|\|Q\chi\|_{L^4(M)}\|\psi\|_{L^4(M)}\|\na\phi\|_{L^{\frac{8}{3}}(M)}.
  \end{split}
 \end{equation}
 Combining with \eqref{ep-0 assumption} this gives
 \begin{equation}
  \|\na\phi\|_{L^{\frac{8}{3}}(M)}\le C\vep_0^{\frac{3}{4}} \|\psi\|_{W^{1,\frac{8}{5}}(M)}
  \le C\vep_0^{\frac{3}{4}}\left(\|Q\chi\|_{L^4(M)}+\|\psi\|_{L^4(M)}\right)\|\na\phi\|_{L^{\frac{8}{3}}(M)}.
 \end{equation}
 Therefore, when $\vep_0$ is sufficiently small, this implies $\na\phi\equiv 0$, that is, $\phi=const.$ Then \eqref{control psi in terms of phi} says that $\psi$ is also trivial.

\end{proof}

\begin{remark}
Observe that although the estimates here are similar to those in the proof of small energy regularities, they come from a different point of view. There we have to take cutoff functions to make the boundary terms vanish in order that the local elliptic estimates are applicable without boundary terms. Here, on the contrary, we rely on the hypothesis that $S$ doesn't admit nontrivial harmonic spinors to obtain the estimate \eqref{control psi in terms of phi} which is a global property.
\end{remark}

%%%%%%%%%%%%%%%%%%%%%%%%%%%%%%%%%%%%%----3----%%%%%%%%%%%%%%%%%%%%%%%%%%%%%%%%%%%%%%%%%%%%%%%%%%%%%%%%%%%%

\section{Critical gravitino and energy-momentum tensor}

In this section we consider the energy-momentum tensor along a solution to \eqref{EL}. We will see that it gives rise to a holomorphic quadratic differential when the gravitino is critical, which is needed for the later analysis.

From now on we assume that the gravitino $\chi$ is also critical for the action functional with respect to variations; that is, for any smooth family $(\chi_t)_t$ of gravitinos with $\chi_0=\chi$, it holds that
\begin{equation}
 \dt{0}\A(\phi,\psi;g,\chi_t)=0.
\end{equation}
One can conclude from this by direct calculation that the \emph{supercurrent} $J=J^\al\otimes e_\al$ vanishes (or see \cite{jost2016note}), where
\begin{equation}
 J^\al=2\langle\phi_*e_\be, e_\be\cdot e_\al\cdot \psi\rangle_{\phi^*h}+|\psi|^2 e_\be\cdot e_\al\cdot \chi^\be.
\end{equation}
Equivalently it can be formulated as
\begin{equation}
 |\psi|^2 e_\be\cdot e_\al\cdot \chi^\be=-2\langle \phi_*e_\be,e_\be\cdot e_\al\cdot \psi\rangle_{\phi^*h}, \qquad \forall \al.
\end{equation}
Recall that $Q\chi=-\frac{1}{2}e_\be\cdot e_\al\cdot \chi^\be\otimes e_\al$. Thus
\begin{equation}\label{reformulation of |psi|Qx}
 |\psi|^2Q\chi=-\frac{1}{2}|\psi|^2 e_\be\cdot e_\al \cdot \chi^\be\otimes e_\al
 =\langle \phi_* e_\be, e_\be\cdot e_\al\cdot \psi\rangle_{\phi^*h}\otimes e_\al.
\end{equation}
It follows that
\begin{equation}
 \begin{split}
  |Q\chi|^2|\psi|^2=\langle\chi, |Q\chi|^2\chi\rangle_{\phi^*h}
  &=\left\langle\chi^\eta\otimes e_\eta, \langle\phi_*e_\be,e_\be\cdot e_\al\cdot\psi\rangle_{\phi^*h}\otimes e_\al \right\rangle_{g_s\otimes g} \\
  &=\langle\chi^\al\otimes\phi_*e_\be, e_\be\cdot e_\al\cdot\psi\rangle_{g_s\otimes \phi^*h} \\
  &=\left\langle e_\al\cdot e_\be\cdot \chi^\al\otimes \phi_* e_\be,\psi\right\rangle_{g_s\otimes\phi^*h} \\
  &=-2\left\langle(\mathds{1}\otimes \phi_*)Q\chi,\psi\right\rangle.
 \end{split}
\end{equation}
Since the Euler--Lagrange equations for $\psi$ are
\begin{equation}\label{EL-psi}
 \D\psi=\frac{1}{3}SR(\psi)+|Q\chi|^2\psi+2(\mathds{1}\otimes\phi_*)Q\chi,
\end{equation}
so if $\psi$ is critical, i.e. the above equation \eqref{EL-psi} holds, then
\begin{equation}
 \langle\psi, \D\psi\rangle=\frac{1}{3}\langle SR(\psi), \psi\rangle=\frac{1}{3}\Rm(\psi).
\end{equation}
Therefore the following relation holds:
\begin{equation}
 \langle\psi, e_2\cdot \widetilde{\na}_{e_2}\psi\rangle=-\langle\psi, e_1\cdot \widetilde{\na}_{e_1}\psi\rangle+\frac{1}{3}\Rm(\psi).
\end{equation}

\begin{lemma}
 For any $\phi$ and $\psi$, and for any $\be$,
 \begin{equation}
 e_\be(|Q\chi|^2|\psi|^2)
  =2\langle \na^s_{e_\be}(e_\al\cdot e_\eta\cdot \chi^\al)\otimes \phi_* e_\eta, \psi\rangle
     + 2|Q\chi|^2\langle\psi,\widetilde{\na}_{e_\be}\psi\rangle.
 \end{equation}
\end{lemma}
\begin{proof}
Since
\begin{equation}
 e_\be(|Q\chi|^2|\psi|^2)=e_\be(|Q\chi|^2)|\psi|^2+2|Q\chi|^2 \langle\psi,\widetilde{\na}_{e_\be}\psi\rangle,
\end{equation}
it suffices to compute $e_{\be}(|Q\chi|^2)|\psi|^2$. Note that
\begin{equation}
 \begin{split}
  e_\be(|Q\chi|^2)
   &=e_\be \langle\chi, Q\chi\rangle=-\frac{1}{2} e_\be\langle\chi^\al, e_\eta\cdot e_\al\cdot\chi^\eta\rangle\\
   &=-\frac{1}{2}\big(\langle\na^s_{e_\be}\chi^\al,e_\eta\cdot e_\al\cdot \chi^\eta\rangle
       +\langle \chi^\al,e_\eta\cdot e_\al \cdot \na^s_{e_\be}\chi^\eta\rangle \big)\\
   &=-\langle \na^s_{e_\be}\chi^\al, e_\eta\cdot e_\al\cdot \chi^\eta\rangle.
 \end{split}
\end{equation}
Therefore, by virtue of \eqref{reformulation of |psi|Qx},
\begin{equation}
 \begin{split}
   e_\be(|Q\chi|^2)|\psi|^2
    &=-\langle \na^s_{e_\be}\chi^\al, |\psi|^2 e_\eta\cdot e_\al\cdot \chi^\eta\rangle
     =2\big\langle \na^s_{e_\be}\chi^\al,\langle\phi_*  e_\eta, e_\eta\cdot e_\al\cdot \psi\rangle_{\phi^*h}\big\rangle_{g_s} \\
    &=2\langle\na^s_{e_\be}\chi^\al\otimes\phi_* e_\eta, e_\eta\cdot e_\al\cdot \psi\rangle
     =2\langle \na^s_{e_\be}(e_\al\cdot e_\eta\cdot \chi^\al)\otimes \phi_* e_\eta, \psi\rangle.
 \end{split}
\end{equation}
The desired equality follows.

\end{proof}

\begin{lemma}
 For any $\phi$ and $\psi$,
 \begin{equation}\label{vanishing mixed term}
 \langle (\mathds{1}\otimes\phi_*)Q\chi, \omega\cdot \psi \rangle=0,
 \end{equation}
 where $\omega=e_1\cdot e_2$ is the volume element.
\end{lemma}
\begin{proof}
 Since
 \begin{equation}
  |\psi|^2 (\mathds{1}\otimes\phi_*)Q\chi
  =-\frac{1}{2}|\psi|^2 e_\eta \cdot e_\al\cdot \chi^\eta\otimes \phi_*e_\al
  =\langle \phi_* e_\eta, e_\eta\cdot e_\al\cdot\psi \rangle_{\phi^*h} \otimes \phi_*e_\al,
 \end{equation}
  We have
 \begin{equation}
  \begin{split}
      |\psi|^2\langle(\mathds{1}\otimes\phi_*)Q\chi, e_1\cdot e_2\cdot\psi\rangle
   =&\left\langle
      \langle \phi_* e_\eta, e_\eta\cdot e_\al\cdot\psi \rangle_{\phi^*h} \otimes \phi_*e_\al, e_1\cdot e_2\cdot\psi
     \right\rangle  \\
   =&\big\langle
     \langle \phi_* e_\eta,e_\eta\cdot e_\al\cdot\psi\rangle_{\phi^*h},
     \langle\phi_* e_\al, e_1\cdot e_2\cdot\psi\rangle_{\phi^*h}
     \big\rangle_{g_s}.
  \end{split}
 \end{equation}
 According to the Clifford relation it holds that
 \begin{equation}
  \begin{split}
   \big\langle
     \langle \phi_* e_\eta,e_\eta\cdot e_\al\cdot\psi\rangle_{\phi^*h},
     &\langle\phi_* e_\al, e_1\cdot e_2\cdot\psi\rangle_{\phi^*h}
     \big\rangle_{g_s} \\
   =&\big\langle
     \langle \phi_* e_\eta,e_2\cdot e_1\cdot e_\eta\cdot e_\al\cdot\psi\rangle_{\phi^*h},
     \langle\phi_* e_\al, \psi\rangle_{\phi^*h}
     \big\rangle_{g_s} \\
   =&\big\langle
     \langle \phi_* e_\eta,e_\eta\cdot e_\al\cdot e_2\cdot e_1\cdot \psi\rangle_{\phi^*h},
    \langle\phi_* e_\al, \psi\rangle_{\phi^*h}
     \big\rangle_{g_s} \\
   =&\big\langle
     \langle \phi_* e_\eta,  e_2\cdot e_1\cdot \psi\rangle_{\phi^*h},
     \langle\phi_* e_\al,e_\al\cdot e_\eta\cdot\psi\rangle_{\phi^*h}
     \big\rangle_{g_s} \\
   =&-\big\langle
     \langle \phi_* e_\al,  e_1\cdot e_2\cdot \psi\rangle_{\phi^*h},
     \langle\phi_* e_\eta,e_\eta\cdot e_\al\cdot\psi\rangle_{\phi^*h}
     \big\rangle_{g_s}.
  \end{split}
 \end{equation}
It follows that
\begin{equation}
 |\psi|^2\langle(\mathds{1}\otimes\phi_*)Q\chi, e_1\cdot e_2\cdot\psi\rangle=0.
\end{equation}
At any point $x\in M$, if $\psi(x)=0$, then \eqref{vanishing mixed term} holds; and if $|\psi(x)|\neq 0$, then by the calculations above \eqref{vanishing mixed term} also holds. This finishes the proof.

\end{proof}
\begin{remark}
More explicitly \eqref{vanishing mixed term} is equivalent to
\begin{equation}\label{vanishing mixed term-local}
 \langle e_1\cdot e_2\cdot\chi^1\otimes\phi_*e_1+\chi^1\otimes\phi_*e_2-\chi^2\otimes\phi_*e_1+e_1\cdot e_2\cdot\chi^2\otimes\phi_*e_2,\psi\rangle=0.
\end{equation}
\end{remark}

From \cite{jost2016note} we know the energy-momentum tensor is given by $T=T_{\al\be}e^\al\otimes e^\be$ where
\begin{equation}\label{energy-momentum}
 \begin{split}
  T_{\al\be}
  =&2\langle \phi_*e_\al,\phi_* e_\be\rangle_{\phi^*h}-|\dd\phi|^2 g_{\al\be}
     +\frac{1}{2}\left\langle\psi,e_\al\cdot\widetilde{\na}_{e_\be}\psi+e_\be\cdot\widetilde{\na}_{e_\al}\psi\right\rangle_{g_s\otimes\phi^*h}
                        -\langle\psi,\D_g\psi\rangle g_{\al\be}\\
   & +\langle e_\eta\cdot e_\al\cdot\chi^\eta\otimes \phi_*e_\be+e_\eta\cdot e_\be\cdot\chi^\eta\otimes\phi_*e_\al, \psi\rangle_{g_s\otimes\phi^*h}
                 +4\langle(\mathds{1}\otimes\phi_*)Q\chi,\psi\rangle g_{\al\be}\\
   &\qquad + |Q\chi|^2 |\psi|^2 g_{\al\be}
                 +\frac{1}{6}\Rm(\psi)g_{\al\be}.
 \end{split}
\end{equation}
Suppose that $(\phi,\psi)$ satisfies the Euler--Lagrange equations \eqref{EL} and that the supercurrent $J$ vanishes. Then
\begin{equation}
 \begin{split}
  T_{\al\be}
  =&2\langle\phi_* e_\al,\phi_* e_\be\rangle-|\dd\phi|^2 g_{\al\be}
    +\frac{1}{2}\langle\psi, e_\al\cdot\widetilde{\na}_{e_\be}\psi+e_\be\cdot\widetilde{\na}_{e_\al}\psi\rangle
    -\frac{1}{2}\langle\psi,\D\psi\rangle g_{\al\be} \\
   &+\langle e_\eta\cdot e_\al\cdot\chi^\eta\otimes \phi_*e_\be+e_\eta\cdot e_\be\cdot\chi^\eta\otimes\phi_*e_\al, \psi\rangle -\langle e_\theta\cdot e_\eta\cdot\chi^\theta\otimes\phi_* e_\eta,\psi\rangle g_{\al\be}.
 \end{split}
\end{equation}
Clearly $T$ is symmetric and traceless. We will show it is also divergence free. Before this we rewrite it into a suitable form. Multiplying $\omega=e_1\cdot e_2$ to both sides of equations \eqref{EL-psi}, we get
\begin{equation}
 e_2\cdot\widetilde{\na}_{e_1}\psi-e_1\cdot \widetilde{\na}_{e_2}\psi
 =\frac{1}{3}\omega\cdot SR(\psi)+|Q\chi|^2\omega\cdot\psi+2\omega\cdot(\mathds{1}\otimes\phi_*)Q\chi.
\end{equation}
Note that the right hand side is perpendicular to $\psi$:
\begin{equation}
 \begin{split}
  \langle\psi, \omega\cdot SR(\psi)\rangle
  &=R_{ijkl}\langle\psi^j,\psi^l\rangle \langle\psi^i,\omega\cdot\psi^k\rangle=0,\\
  |Q\chi|^2\langle\psi,\omega\cdot\psi\rangle
  &=0,\\
  \langle2\omega\cdot(\mathds{1}\otimes\phi_*)Q\chi,\psi\rangle
   &=-2\langle(\mathds{1}\otimes\phi_*)Q\chi,\omega\cdot\psi\rangle=0.
 \end{split}
\end{equation}
Hence $\langle\psi, e_2\cdot \widetilde{\na}_{e_1}\psi\rangle-\langle\psi, e_1\cdot \widetilde{\na}_{e_2}\psi\rangle=0$. Consequently,
\begin{equation}
 \frac{1}{2}\langle\psi, e_\al\cdot\widetilde{\na}_{e_\be}\psi+e_\be\cdot\widetilde{\na}_{e_\al}\psi\rangle
 =\langle\psi, e_\al\cdot \widetilde{\na}_{e_\be}\psi\rangle.
\end{equation}
Moreover, by \eqref{vanishing mixed term-local},
\begin{equation}
 \begin{split}
  \langle e_\eta\cdot e_1\cdot\chi^\eta\otimes
  &\phi_*e_2,\psi\rangle
  -\langle e_\eta\cdot e_2\cdot\chi^\eta\otimes\phi_* e_1,\psi\rangle \\
  =&\langle-\chi^1\otimes\phi_*e_2-e_1\cdot e_2\cdot\chi^2\otimes\phi_*e_2-e_1\cdot e_2\cdot\chi^1\otimes\phi_*e_1+\chi^2\otimes\phi_*e_1,\psi\rangle=0.
 \end{split}
\end{equation}
Therefore, we can put the energy-momentum tensor into the following form:
\begin{equation}\label{energy-momentum-2}
 \begin{split}
  T_{\al\be}
  =&2\langle\phi_* e_\al,\phi_* e_\be\rangle-|\dd\phi|^2 g_{\al\be}
    +\langle\psi, e_\al\cdot\widetilde{\na}_{e_\be}\psi\rangle
    -\frac{1}{2}\langle\psi,\D\psi\rangle g_{\al\be} \\
   &+2\langle e_\eta\cdot e_\al\cdot\chi^\eta\otimes \phi_*e_\be, \psi\rangle
    -\langle e_\theta\cdot e_\eta\cdot\chi^\theta\otimes\phi_* e_\eta,\psi\rangle g_{\al\be}.
 \end{split}
\end{equation}
This form relates closely to the energy-momentum tensor for
Dirac-harmonic maps in \cite[Section 3]{chen2006dirac} and that for
Dirac-harmonic maps with curvature term in \cite[Section
4]{jost2015geometric}, which also have the following nice
properties. Such computations have been provided in \cite[Section
3]{branding2015some}, but since certain algebraic aspects are
different here, we need to spell out the computations in detail.
\begin{prop}\label{properties of energy-momentum tensor}
 Let $(\phi,\psi,\chi)$ be critical. Then the tensor $T$ given by \eqref{energy-momentum} or equivalently \eqref{energy-momentum-2} is symmetric, traceless, and covariantly conserved.
\end{prop}

\begin{proof}
 It remains to show that $T$ is covariantly conserved. Let $x\in M$ and take the normal coordinate at $x$ such that $\na e_\al(x)=0$. We will show that $\na_{e_\al}T_{\al\be}(x)=0$. At the point $x$, making use of the Euler--Lagrange equations, one can calculate as follows.
 \begin{itemize}
  \item
        \begin{equation}
        \begin{split}
         \na_{e_\al}(2\langle\phi_* e_\al,
           &\phi_* e_\be\rangle
            -2|\dd\phi|^2 g_{\al\be} )  \\
          =& 2\langle\na_{e_\al}(\phi_* e_\al),\phi_* e_\be\rangle+2\langle\phi_* e_\al,\na_{e_\al}(\phi_* e_\be)\rangle
              -2\langle\phi_* e_\al, \na_{e_\be}(\phi_* e_\al)\rangle \\
          % & (\textnormal{the last two summands are equal since at the point under consideration,} \\
           %&   \quad   \na_{e_\al}\phi_* e_\be=\na_{e_\be}\phi_* e_\al)  \\
          =& 2\langle \tau(\phi),\phi_* e_\be \rangle \\
          =& \left\langle \Rm(\psi, e_\al\cdot\psi)\phi_* e_\al, \phi_* e_\be\right\rangle
                  -\frac{1}{6}\langle S\na R(\psi),\phi_* e_\be\rangle   \\
           & \quad -2\langle \na^s_{e_\al}(e_\eta\cdot e_\al\cdot \chi_\eta)\otimes \phi_* e_\be,\psi\rangle
                   -2\langle e_\eta\cdot e_\al\cdot\chi_\eta\otimes\phi_* e_\be, \widetilde{\na}_{e_\al}\psi\rangle.
        \end{split}
        \end{equation}
  \item
        \begin{equation}
         \begin{split}
          \na_{e_\al}(\langle\psi,\;
             &e_\al\cdot\widetilde{\na}_{e_\be}\psi\rangle
             -\langle\psi,\D\psi\rangle g_{\al\be}) \\
            =&\langle\widetilde{\na}_{e_\al}\psi, e_\al\cdot\widetilde{\na}_{e_\be}\psi\rangle+\langle\psi, e_\al\cdot\widetilde{\na}_{e_\al}\widetilde{\na}_{e_\be}\psi\rangle
                -\langle\widetilde{\na}_{e_\be}\psi,\D\psi\rangle-\langle\psi,\widetilde{\na}_{e_\be}\D\psi\rangle \\
            =&-\langle\D\psi, \widetilde{\na}_{e_\be}\psi\rangle+\langle\psi, \D\widetilde{\na}_{e_\be}\psi\rangle
                -\langle\widetilde{\na}_{e_\be}\psi,\D\psi\rangle-\langle\psi,\widetilde{\na}_{e_\be}\D\psi\rangle \\
            =&-2\langle\D\psi,\widetilde{\na}_{e_\be}\psi\rangle+\langle\psi,\D\widetilde{\na}_{e_\be}\psi-\widetilde{\na}_{e_\be}\D\psi\rangle.
         \end{split}
        \end{equation}
        Note that
        \begin{equation}
         \begin{split}
          \D\widetilde{\na}_{e_\be}\psi-\widetilde{\na}_{e_\be}\D\psi
          =& e_\al\cdot \Ric^{S}(e_\al,e_\be)\psi+ \Rm(\phi_* e_\al, \phi_* e_\be)\\
          =&\frac{1}{2}\Ric(e_\be)\psi+\Rm(\phi_* e_\al, \phi_* e_\be) e_\al\cdot\psi,
         \end{split}
        \end{equation}
        and that ${\langle\psi, \Ric(e_\be)\psi\rangle =0}$.
        Hence one has
        \begin{equation}
         \begin{split}
          \na_{e_\al}(\langle\psi,\;
             & e_\al\cdot\widetilde{\na}_{e_\be}\psi\rangle
                 -\langle\psi,\D\psi\rangle g_{\al\be} )\\
            =&-2\langle |Q\chi|^2\psi+\frac{1}{3}SR(\psi)+2(\mathds{1}\otimes\phi_*)Q\chi, \widetilde{\na}_{e_\be}\psi\rangle  \\
             &\quad  +\langle\psi, \Rm(\phi_* e_\al, \phi_* e_\be) e_\al\cdot\psi\rangle\\
            =& -2|Q\chi|^2\langle\psi, \widetilde{\na}_{e_\be}\psi\rangle -\frac{2}{3}\langle SR(\psi),\widetilde{\na}_{e_\be}\psi\rangle
               -4\langle (\mathds{1}\otimes \phi_*)Q\chi, \widetilde{\na}_{e_\be}\psi\rangle  \\
             &\quad -\langle\Rm(\psi, e_\al\cdot\psi)\phi_* e_\al, \phi_* e_\be\rangle .
         \end{split}
        \end{equation}
   \item
         \begin{equation}
          \begin{split}
           \na_{e_\al}\left(\frac{1}{6}\Rm(\psi)g_{\al\be}\right)
           =\frac{1}{6}\langle S\na R(\psi), \phi_* e_\be\rangle+\frac{2}{3}\langle SR(\psi),\widetilde{\na}_{e_\be}\psi\rangle.
          \end{split}
         \end{equation}
   \item
         \begin{equation}
          \begin{split}
           \na_{e_\al}\big(
                & 2\langle e_\eta\cdot e_\al\cdot\chi_\eta\otimes\phi_* e_\be,\psi\rangle
                  -\delta_{\al\be}\langle e_\eta\cdot e_\eta\cdot\chi_\eta\otimes \phi_* e_\eta, \psi\rangle\big) \\
                &=2\langle\na^s_{e_\al}(e_\eta\cdot e_\al\cdot \chi_\eta)\otimes\phi_* e_\be, \psi\rangle
                  +2\langle e_\eta\cdot e_\al\cdot \chi_\eta\otimes\na_{e_\al}(\phi_* e_\be), \psi\rangle \\
                &\qquad +2\langle e_\eta\cdot e_\al\cdot \chi_\eta\otimes\phi_* e_\be, \widetilde{\na}_{e_\al}\psi\rangle
                   -\na_{e_\be}\big(\langle e_\eta\cdot e_\al\cdot\chi_\eta\otimes \phi_* e_\al,\psi \rangle\big).
          \end{split}
         \end{equation}
 \end{itemize}
 Summarize these terms and use the previous lemmata to get
 \begin{equation}
  \begin{split}
   \na_{e_\al} T_{\al\be}
    &= -2|Q\chi|^2\langle\psi, \widetilde{\na}_{e_\be}\psi\rangle -4\langle (\mathds{1}\otimes \phi_*)Q\chi, \widetilde{\na}_{e_\be}\psi\rangle  \\
    &   \qquad+2\langle e_\eta\cdot e_\al\cdot \chi_\eta\otimes\na_{e_\al}(\phi_* e_\be), \psi\rangle
         -\na_{e_\be}\big(\langle e_\eta\cdot e_\al\cdot\chi_\eta\otimes \phi_* e_\al,\psi \rangle \\
    &= 2\langle\na^s_{e_\be}(e_\al\cdot e_\eta \cdot\chi_\eta)\otimes\phi_* e_\eta, \psi\rangle
        -\na_{e_\be}\big(\langle e_\eta\cdot e_\al\cdot \chi_\eta \otimes \phi_* e_\al, \psi\rangle\big) \\
    &\qquad +2\langle e_\eta\cdot e_\al\cdot \chi_\eta\otimes\phi_* e_\al,\widetilde{\na}_{e_\be}\psi\rangle
             +2\langle e_\eta\cdot e_\al\cdot \chi_\eta\otimes\na_{e_\be}(\phi_* e_\al), \psi\rangle \\
    &\qquad -\na_{e_\be}\big(\langle e_\eta\cdot e_\al\cdot \chi_\eta \otimes \phi_* e_\al, \psi\rangle\big) \\
    &=0.
  \end{split}
 \end{equation}
 This accomplishes the proof.

\end{proof}

As in the harmonic map case, such a 2-tensor then corresponds to a holomorphic quadratic differential on $M$. For the case of Dirac-harmonic maps (with or without curvature terms), see \cite{chen2006dirac, jost2015geometric} and  \cite{branding2015some}. More precisely, in a local isothermal coordinate $z=x+iy$, set
\begin{equation}
 T(z)\dd z^2\coloneqq(T_{11}-iT_{12})(\dd x+i\dd y)^2,
\end{equation}
with $T_{11}$ and $T_{12}$ now being the coefficients of the energy-momentum tensor $T$ in the local coordinate, that is,
\begin{equation}
 \begin{split}
  T_{11}
  &=\left|\frac{\p\phi}{\p x}\right|^2-\left|\frac{\p\phi}{\p y}\right|^2
   +\frac{1}{2}\left(\langle\psi,\gamma(\p_x)\widetilde{\na}_{\p_x}\psi\rangle-\langle\psi,\gamma(\p_y)\widetilde{\na}_{\p_y}\psi\rangle\right)
    +F_{11},\\
  T_{12}
  &=\left\langle\frac{\p\phi}{\p x},\frac{\p\phi}{\p y}\right\rangle_{\phi^*h}
      +\langle\psi,\gamma(\p_x)\widetilde{\na}_{\p_y}\psi\rangle+F_{12}.
 \end{split}
\end{equation}
Here we have abbreviated the gravitino terms as $F_{\al\be}$'s:
\begin{equation}
 \begin{split}\label{abbreviation F}
  F_{11}&=2\langle-\chi^x\otimes\phi_*(\p_x)-\gamma(\p_x)\gamma(\p_y)\chi^y\otimes\phi_*(\p_x),\psi\rangle
           +2\langle(\mathds{1}\otimes\phi_*)Q\chi, \psi\rangle g(\p_x,\p_x), \\
  F_{12}&=2\langle-\chi^x\otimes\phi_*(\p_y)-\gamma(\p_x)\gamma(\p_y)\chi^y\otimes \phi_*(\p_y),\psi\rangle,
 \end{split}
\end{equation}
where $\chi=\chi^x\otimes \p_x+\chi^y\otimes\p_y$ in a local chart.
\begin{prop}\label{holomorphicity of T(z)}
 The quadratic differential $T(z)\dd z^2$ is well-defined and holomorphic.
\end{prop}

\begin{proof}
 The well-definedness is straightforward and the holomorphicity follows from Proposition \ref{properties of energy-momentum tensor}.

\end{proof}

%%%%%%%%%%%%%%%%%%%%%%%%%%%%%%%%%%----4----%%%%%%%%%%%%%%%%%%%%%%%%%%%%%%%%%%%%%%%%%%%%%%%%%%%%%%%%%%%%%

\section{Pohozaev identity and removable singularities}

In this section we show that a solution of \eqref{EL} with finite
energy admits no isolated poles, provided that the gravitino is
critical. As the singularities under consideration are isolated, we
can locate the solution on the punctured Euclidean unit disk
$B_1^*\equiv B_1\backslash \{0 \}$. Using the quadratic holomorphic
differential derived in the previous section, we obtain the Pohozaev
type formulae containing gravitino terms in Theorem \ref{Pohozaev
  Identity}. When the gravitino vanishes, they will reduce to the
Pohozaev identities for Dirac-harmonic maps with curvature term, see
e.g. \cite[Lemma 5.3]{jost2015geometric} and also \cite[Lemma
3.11]{branding2016energy} where a somewhat different identity is derived. %Here we restate Theorem \ref{Pohozaev Identity}.

Recall that the $F_{\al\be}$'s are given in \eqref{abbreviation F} and they can be controlled via Young inequality by
\begin{equation}
 |F_{\al\be}|\le C|\na\phi||\psi||\chi|\le C(|\na\phi|^2+|\psi|^4+|\chi|^4).
\end{equation}

\begin{proof}[Proof of Theorem \ref{Pohozaev Identity}]

 By definition we have
 \begin{equation}
  |T(z)|\le C\left(|\na\phi|^2+|\widetilde{\na}\psi||\psi|+|F_{\al\be}|\right).
 \end{equation}
 Note that $|\na\psi|\le C(|\na^s\psi|+|\psi||\na\phi|)$. Apply the Young inequality once again to obtain
 \begin{equation}
  |T(z)|\le C\left(|\na\phi|^2+|\psi|^4+|\na^s\psi|^{\frac{4}{3}}+|\chi|^4\right).
 \end{equation}
 From the initial assumptions we known that $\phi\in W^{1,2}(B_1^*, N)$, $\psi\in L^4(B_1^*)$ and $\chi$ is smooth in $B_1$, thus by Theorem \ref{weak solutions extension}, $(\phi,\psi)$ is actually a weak solution on the whole disk $B_1$. Using the ellipticity of the Dirac operator, $\psi$ belongs to $W^{1,\frac{4}{3}}_{loc}(B_1)$. Therefore $|T(z)|$ is integrable on the disk $B_r$ for any $r<1$. Recall from Proposition \ref{holomorphicity of T(z)} that $T(z)$ is a holomorphic function defined on the punctured disk. Hence, it has a pole at the origin of order at most one. In particular, $z T(z)$ is holomorphic in the whole disk. Then by Cauchy theorem, for any $0<r<1$, it holds that $\int_{|z|=r} z T(z)\dd z=0$. One can compute that in polar coordinate $z=re^{i\theta}$,
\begin{equation}\label{real part of zzq}
 \begin{split}
  \frac{1}{r^2}Re(z^2 T(z))
  =&\left|\frac{\p\phi}{\p r}\right|^2-\frac{1}{r^2}\left|\frac{\p\phi}{\p\theta}\right|^2
   +\frac{1}{2}\left(\left\langle\psi,\gamma(\p_r)\na_{\p_r}\psi\right\rangle
                     -\left\langle\psi,\frac{1}{r^2}\gamma(\p_\theta)\na_{\p_\theta}\psi\right\rangle\right) \\
   &+F_{11}\cos{2\theta}+F_{12}\sin{2\theta}.
 \end{split}
\end{equation}
The identity $\langle\psi,\D\psi\rangle=\Rm(\psi)/3$ along a critical $\psi$ implies
\begin{equation}
 \begin{split}
  \frac{1}{2}\left( \left\langle\psi,\gamma(\p_r)\na_{\p_r}\psi\right\rangle
                     -\left\langle\psi,\frac{1}{r^2}\gamma(\p_\theta)\na_{\p_\theta}\psi\right\rangle \right)
  &=\langle\psi,\gamma(\p_r)\na_{\p_r}\psi\rangle-\frac{1}{6}\Rm(\psi) \\
  &=-\left\langle\psi,\frac{1}{r^2}\gamma(\p_\theta)\na_{\p_\theta}\psi\right\rangle+\frac{1}{6}\Rm(\psi).
 \end{split}
\end{equation}
 Finally, it suffices to note that
 \begin{equation}
  Im\left(\int_{|z|=r} z T(z)\dd z\right)=r \int_0^{2\pi} Re(z^2 T(z))\dd \theta.
 \end{equation}
\end{proof}

Integrating \eqref{Pohozaev identity} with respect to the radius, we get
\begin{equation}
 \begin{split}
 \int_{B_1} \left|\frac{\p\phi}{\p r}\right|^2-\frac{1}{r^2}\left|\frac{\p\phi}{\p\theta}\right|^2 \dd x
   =&\int_{B_1} -\langle\psi, \gamma(\p_r)\na_{\p_r}\psi\rangle+\frac{1}{6}\Rm(\psi) -(F_{11}\cos{2\theta}+F_{12}\sin{2\theta})\dd x \\
   =&\int_{B_1} \left\langle\psi,\frac{1}{r^2}\gamma(\p_\theta)\na_{\p_\theta}\psi\right\rangle-\frac{1}{6}\Rm(\psi)- (F_{11}\cos{2\theta}+F_{12}\sin{2\theta})\dd x.
 \end{split}
\end{equation}
Meanwhile note that in polar coordinate $(r,\theta)$,
\begin{equation}\label{gradient in polar coordinate}
 |\na\phi|^2=\left|\frac{\p\phi}{\p r}\right|^2+\frac{1}{r^2}\left|\frac{\p\phi}{\p\theta}\right|^2.
\end{equation}
This can be combined with Theorem \ref{Pohozaev Identity} to give estimates on each component of the gradient of the map $\phi$; in particular,
\begin{equation}\label{angle derivative}
\int_{B_1}\frac{1}{r^2}\left|\frac{\p\phi}{\p\theta}\right|^2\dd x
 =\frac{1}{2} \int_{B_1} |\na\phi|^2 +\langle\psi,\gamma(\p_r)\na_{\p_r}\psi\rangle-\frac{1}{6}\Rm(\psi)
    +F_{11}\cos{2\theta}+ F_{12}\sin{2\theta} \dd x.
\end{equation}

Next we consider the isolated singularities of a solution. We show they are removable provided the gravitino is critical and does not have a singularity there, and the energy of the solution is finite. Different from Dirac-harmonic maps in \cite[Theorem 4.6]{chen2006dirac} and those with curvature term in \cite[Theorem 6.1]{jost2015geometric} (ses also \cite[Theorem 3.12]{branding2016energy}), we obtain this result using the regularity theorems of weak solutions. Thus we have to show first that weak solutions can be extended over an isolated point in a punctured neighborhood. This is achieved in the Appendix.

\begin{thm}[restate=Removable Singularity, label=removable singularities]
 \textnormal{\textbf{(Removable singularity.)}}
 Let $(\phi,\psi)$ be a smooth solution defined on the punctured disk $B_1^*\equiv B_1\backslash\{0\}$. If $\chi$ is a smooth critical gravitino on $B_1$ and if $(\phi,\psi)$ has finite energy on $B_1^*$, then $(\phi,\psi)$ extends to a smooth solution on $B_1$.
\end{thm}

\begin{proof}
 From Theorem \ref{weak solutions extension} in the Appendix we know that $(\phi,\psi)$ is also a weak solution on the whole disk $B_1$. By taking a smaller disc centered at the origin and rescaling as above, one may assume that~$E(\phi,\psi;B_1)$ and $\|\chi\|_{W^{1,\frac{4}{3}}(B_1)}$ are sufficiently small. From the result in \cite{jost2016regularity} we then see that $(\phi,\psi)$ is actually smooth in $B_{1/2}(0)$. In addition to the assumption, we see that it is a smooth solution on the whole disk.

\end{proof}

%%%%%%%%%%%%%%%%%%%%%%%%%%%%%%----5----%%%%%%%%%%%%%%%%%%%%%%%%%%%%%%%%%%%%%%%%%%%%%%%%%%%%%%%%%%%%%%%
\section{Energy identity}

In this section we consider the compactness of the critical points space, i.e. the space of solutions of \eqref{EL}. In the end we will prove the main result, the energy identities in Theorem \ref{energy identity thm}. As in \cite[Lemma 3.2]{zhao2006energy} we establish the following estimate for $\psi$ on annulus domains, which is useful for the proof of energy identities. Let $0<2r_2< r_1<1$.
\begin{lemma}\label{estimates of psi on annulus}
 Let $\psi$ be a solution of \eqref{EL-psi-extrinsic} defined on $A_{r_2,r_1}\equiv B_{r_1}\backslash B_{r_2}$. Then
 \begin{equation}
  \begin{split}
   \|\widetilde{\na}\psi\|_{L^{\frac{4}{3}}(B_{r_1}\backslash B_{2r_2})}
   &+\|\psi\|_{L^{4}(B_{r_1}\backslash B_{2r_2})} \\
   \le& C_0\left(|A|\|\na\phi\|_{L^2(A_{r_2,r_1})}+\|Q\chi\|^2_{L^4(A_{r_2,r_1})}
               +|A|^2\|\psi\|^2_{L^4(A_{r_2,r_1})}\right)\|\psi\|_{L^4(A_{r_2,r_1})}\\
      & +C\|Q\chi\|_{L^4(A_{r_2,r_1})}\|\na\phi\|_{L^2(A_{r_2,r_1})}
        +C\|\psi\|_{L^4(B_{2r_2}\backslash B_{r_2})} \\
      & +C r_1^{\frac{3}{4}}\|\widetilde{\na}\psi\|_{L^{\frac{4}{3}}(\p B_{r_1})}
        +C r_1^{\frac{1}{4}}\|\psi\|_{L^4(\p B_{r_1})},
  \end{split}
 \end{equation}
 where $C_0\ge 1$ is a universal constant which doesn't depend on $r_1$ and $r_2$.
\end{lemma}
\begin{proof}
 Under a rescaling by $1/r_1$, the domain $A_{r_2,r_1}$ changes to $B_1\backslash B_{r_0}$ where $r_0=r_2/r_1$. By rescaling invariance it suffices to prove it on $B_1\backslash B_{r_0}$. Choose a cutoff function $\eta_{r_0}$ such that $\eta_{r_0}=1$ in $B_1\backslash B_{2r_0}$, $\eta_{r_0}=0$ in $B_{r_0}$, and that $|\na\eta_{r_0}|\le C/r_0$.  Similarly as in the previous sections, the equations for $\eta_{r_0}\psi$ read
 \begin{equation}
  \begin{split}
   \pd\left(\eta_{r_0}\psi^i\right)
   =&\eta_{r_0}\left(-A^i_{jk}\na\phi^j\cdot\psi^k+|Q\chi|^2\psi^i
                     +\frac{1}{3}A^i_{jm}A^m_{kl}\left(\langle\psi^k,\psi^l\rangle\psi^j
                                                         -\langle\psi^j,\psi^k\rangle\psi^l\right)\right)\\
    & -\eta_{r_0} e_\al\cdot\na\phi^i\cdot\chi^\al
     +\na\eta_{r_0}\cdot\psi^i.
  \end{split}
 \end{equation}
 Using \cite[Lemma 4.7]{chen2006dirac}, we can estimate
 \begin{equation}
  \begin{split}
   \|\eta_{r_0}\psi\|_{W^{1,\frac{4}{3}}(B_1)}
   \le& C'_0|A|\big\|\eta_{r_0}|\na\phi||\psi|\big\|_{L^{\frac{4}{3}}(B_1)}
         +C'_0\big\|\eta_{r_0}|Q\chi|^2|\psi|\big\|_{L^{\frac{4}{3}}(B_1)}
         +C'_0|A|^2\big\|\eta_{r_0}|\psi|^3\big\|_{L^{\frac{4}{3}}(B_1)} \\
      & +C'_0\big\|\eta_{r_0}|\na\phi||Q\chi|\big\|_{L^{\frac{4}{3}}(B_1)}
         +C'_0\big\||\na\eta_{r_0}||\psi|\big\|_{L^{\frac{4}{3}}(B_1)}
         +C'_0\|\eta_{r_0}\psi\|_{W^{1,\frac{4}{3}}(\p B_1)},
  \end{split}
 \end{equation}
 where the constant $C'_0$ is also from \cite[Lemma 4.7]{chen2006dirac}. This implies that
 \begin{equation}
  \begin{split}
   \|\psi\|_{W^{1,\frac{4}{3}}(B_1\backslash B_{2r_0})}
   \le&2 C'_0|A|\|\na\phi\|_{L^2(B_1\backslash B_{r_0})}\|\psi\|_{L^4(B_1\backslash B_{r_0})}
        +C'_0\|Q\chi\|^2_{L^4(B_1\backslash B_{r_0})}\|\psi\|_{L^4(B_1\backslash B_{r_0})} \\
      & +C'_0|A|^2\|\psi\|^3_{L^4(B_1\backslash B_{r_0})}
        +C'_0\|Q\chi\|_{L^4(B_1\backslash B_{r_0})}\|\na\phi\|_{L^2(B_1\backslash B_{r_0})}\\
      & +C'_0\|\na\eta_{r_0}\|_{L^2(B_{2r_0}\backslash B_{r_0})}\|\psi\|_{L^4(B_{2r_0}\backslash B_{r_0})}
        +C'_0\|\eta_{r_0}\psi\|_{W^{1,\frac{4}{3}}(\p B_1)} \\
   \le 2C'_0 &\left(|A|\|\na\phi\|_{L^2(B_1\backslash B_{r_0})}+\|Q\chi\|^2_{L^4(B_1\backslash B_{r_0})}
               +|A|^2\|\psi\|^2_{L^4(B_1\backslash B_{r_0})}\right)\|\psi\|_{L^4(B_1\backslash B_{r_0})}\\
       +&C'_0\|Q\chi\|_{L^4(B_1\backslash B_{r_0})}\|\na\phi\|_{L^2(B_1\backslash B_{r_0})}
        +C'_0\|\psi\|_{L^4(B_{2r_0}\backslash B_{r_0})}
        +C'_0\|\eta_{r_0}\psi\|_{W^{1,\frac{4}{3}}(\p B_1)} .
  \end{split}
 \end{equation}
 Using the Sobolev embedding theorem, we obtain the estimate on $B_1\backslash B_{r_0}$, and scaling back, we get the desired result with $C_0=2C'_0$.

\end{proof}

Thanks to the invariance under rescaled conformal transformations, the estimate in Lemma \ref{estimates of psi on annulus} can be applied to any conformally equivalent domain, in particular we will apply it on cylinders later.

Similarly we can estimate the energies of the map $\phi$ satisfying \eqref{EL} on the annulus domains, in the same flavor as for Dirac-harmonic maps, see e.g. \cite[Lemma 3.3]{zhao2006energy}.

\begin{lemma}\label{estimates of phi on annulus}
 Let $(\phi,\psi)$ be a solution of \eqref{EL} defined on $A_{r_2,r_1}$ with critical gravitino. Then
  \begin{equation}
  \begin{split}
   \int_{B_{r_1}\backslash B_{r_2}} |\na\phi|^2\dd x
   \le& C\int_{B_{r_1}\backslash B_{r_2}} |A|^2|\psi|^4+|\widetilde{\na}\psi|^{\frac{4}{3}}+|Q\chi|^2|\psi|^2\dd x \\
      & +C\int_{\p(B_{r_1}\backslash B_{r_2})} (q-\phi) \left(\langle V,\frac{\p}{\p r}\rangle-\frac{\p\phi}{\p r}\right) \dd s \\
     +C_1\sup_{B_{r_1}\backslash B_{r_2}}& |q-\phi| \int_{B_{r_1}\backslash B_{r_2}} |A|^2|\na\phi|^2
          +|A|(|A|+|\na A|)|\psi|^4+|\psi|^2|Q\chi|^2\dd x. \\
  \end{split}
 \end{equation}
 Here $C_1\ge1$ is some universal constant.
\end{lemma}
\begin{proof}
 Make a rescaling as in Lemma \ref{estimates of psi on annulus}. Choose a function $q(r)$ on $B_1$ which is piecewise linear in $\log r$ with
 \begin{equation}
  q(\frac{1}{2^m})=\frac{1}{2\pi}\int_0^{2\pi} \phi(\frac{1}{2^m},\theta)\dd\theta,
 \end{equation}
 for $r_0\le 2^{-m}\le 1$, and $q(r_0)$ is defined to be the average of $\phi$ on the circle of radius $r_0$. Then $q$ is harmonic in $A_{m}\coloneqq \{\frac{1}{2^m}<r<\frac{1}{2^{m-1}}\}\subset B_1\backslash B_{r_0}$ and in the annulus near the boundary $\{x\in\R^2\big||x|=r_0\}$. Note that
 \begin{equation}
  \Delta(q-\phi)=-\Delta\phi=-A(\phi)(\na\phi,\na\phi)+\diverg V-f,
 \end{equation}
 where $V$ is given by \eqref{def of V} and $f$ is an abbreviation for
 \begin{equation}
  f^i\equiv A^i_{jm} A^m_{kl}\langle\psi^j,\na\phi^k\cdot\psi^l\rangle
         +Z^i(A,\na A)_{jklm}\langle\psi^j,\psi^l\rangle\langle\psi^k,\psi^m\rangle
         -A^i_{jk}\langle V^j,\na\phi^k\rangle.
 \end{equation}
 Using Green's formula we get
 \begin{equation}
  \begin{split}
   \int_{B_1\backslash B_{r_0}}|\dd q-\dd\phi|^2\dd x
   =& -\int_{B_1\backslash B_{r_0}} (q-\phi)\Delta(q-\phi)\dd x
         +\int_{\p(B_1\backslash B_{r_0})} (q-\phi)\frac{\p}{\p r}(q-\phi) \dd s.
  \end{split}
 \end{equation}
 Since $q(r_0)$ is the average of $\phi$ over $\p B_{r_0}$ we see that
 \begin{equation}
  \int_{\p(B_1\backslash B_{r_0})} (q-\phi)\frac{\p}{\p r}(q-\phi) \dd s
  =-\int_{\p(B_1\backslash B_{r_0})} (q-\phi)\frac{\p\phi}{\p r} \dd s.
 \end{equation}
 By the equation of $(q-\phi)$,
 \begin{equation}
  \begin{split}
   -\int_{B_1\backslash B_{r_0}} (q-\phi)\Delta(q-\phi)\dd x
   &=\int_{B_1\backslash B_{r_0}} (q-\phi)\left( A(\phi)(\na\phi,\na\phi)+f\right)-(q-\phi)\diverg V \dd x \\
   &=\int_{B_1\backslash B_{r_0}} (q-\phi)\left( A(\phi)(\na\phi,\na\phi)+f\right)+\langle\na(q-\phi),V\rangle \dd x\\
   &\qquad     +\int_{\p(B_1\backslash B_{r_0})} (q-\phi)\langle V,\frac{\p}{\p r}\rangle \dd s.
  \end{split}
 \end{equation}
 These together imply that
 \begin{equation}
  \begin{split}
   \int_{B_1\backslash B_{r_0}}|\dd q-\dd\phi|^2\dd x
  \le& \int_{B_1\backslash B_{r_0}} 2(q-\phi)\left( A(\phi)(\na\phi,\na\phi)+f\right)+|V|^2\dd x \\
     & +\int_{\p(B_1\backslash B_{r_0})} 2(q-\phi) \left(\langle V,\frac{\p}{\p r}\rangle-\frac{\p\phi}{\p r}\right)\dd s.
  \end{split}
 \end{equation}
 Recall the Pohozaev formulae \eqref{Pohozaev identity} or its consequence \eqref{angle derivative}, and note that they hold also on the annulus domains. Note also that
 \begin{equation}
  \int_{B_1\backslash B_{r_0}} |\dd q-\dd\phi|^2\dd x
  \ge \int_{B_1\backslash B_{r_0}} \frac{1}{r^2}\left|\frac{\p\phi}{\p\theta}\right|^2\dd x.
 \end{equation}
 Therefore we get
 \begin{equation}
  \begin{split}
   &\frac{1}{2}\int_{B_1\backslash B_{r_0}} |\na\phi|^2+\langle\psi,\gamma(\p_r)\widetilde{\na}_{\p_r}\psi\rangle
                  -\frac{1}{6}\Rm(\psi)+ F_{11}\cos{2\theta}+F_{12}\sin{2\theta} \dd x \\
  \le& \int_{B_1\backslash B_{r_0}} 2(q-\phi)\left( A(\phi)(\na\phi,\na\phi)+f\right)+|V|^2\dd x
      +\int_{\p(B_1\backslash B_{r_0})} 2(q-\phi) \left(\langle V,\frac{\p}{\p r}\rangle-\frac{\p\phi}{\p r}\right)\dd s.
  \end{split}
 \end{equation}
 From this it follows that
 \begin{equation}
  \begin{split}
   \int_{B_1\backslash B_{r_0}} |\na\phi|^2\dd x
   \le& \int_{B_1\backslash B_{r_0}} |A|^2|\psi|^4+|\widetilde{\na}\psi|^{\frac{4}{3}}+32|Q\chi|^2|\psi|^2\dd x \\
   &+\int_{\p(B_1\backslash B_{r_0})} 8(q-\phi) \left(\langle V,\frac{\p}{\p r}\rangle-\frac{\p\phi}{\p r}\right)\dd s\\
   +16\sup_{B_1\backslash B_{r_0}}&|q-\phi| \int_{B_1\backslash B_{r_0}} |A|^2|\na\phi|^2
          +|A|(|A|+|\na A|)|\psi|^4+|\psi|^2|Q\chi|^2\dd x. \\
  \end{split}
 \end{equation}
 Then we rescale back to $A_{r_2,r_1}$. The universal constant $C_1$ can be taken to be $16$, for instance.

\end{proof}

Finally we can show the energy identities, Theorem \ref{energy identity thm}. The corresponding ones for Dirac-harmonic maps with curvature term were obtained in \cite{jost2015geometric}, following the scheme of \cite{ding1995energy, chen2005regularity} and using a method which is based on a type of three circle lemma. Here we apply a method in the same spirit as those in \cite{ye1994gromov, zhao2006energy}. Since we have no control of higher derivatives of gravitinos, the strong convergence assumption on gravitinos is needed here. We remark that the Pohozaev type identity established in Theorem \ref{Pohozaev Identity} is crucial in the proof of this theorem.

\begin{proof}[Proof of Theorem \ref{energy identity thm}]

 The uniform boundedness of energies implies that there is a subsequence converging weakly in $W^{1,2}\times L^4$ to a limit $(\phi,\psi)$ which is a weak solution with respect to $\chi$. Also the boundedness of energies implies that the blow-up set $\mathcal{S}$ consists of only at most finitely many points (possibly empty). If $\mathcal{S}=\emptyset$, then the sequence converges strongly and the conclusion follows directly. Now we assume it is not empty, say $\mathcal{S}=\{p_1,\dots,p_I\}$. Moreover, using the small energy regularities and compact Sobolev embeddings, by a covering argument similar to that in \cite{sacks1981existence} we see that there is a subsequence converging strongly in the $W^{1,2}\times L^4$-topology on the subset $(M\backslash \cup_{i=1}^I B_\delta(p_i))$ for any $\delta>0$.

 When the limit gravitino $\chi$ is smooth, by the regularity theorems in \cite{jost2016regularity} together with the removable singularity theorem \ref{removable singularities} we see that $(\phi,\psi)$ is indeed a smooth solution with respect to~$\chi$.

 Since $M$ is compact and blow-up points are only finitely many, we can find small disks $B_{\delta_i}$ being small neighborhood of each blow-up point $p_i$ such that $B_{\delta_i}\cap B_{\delta_j}=\emptyset$ whenever $i\neq j$ and on $M\backslash \bigcup_{i=1}^I B_{\delta_i}$, the sequence~$(\phi_k,\psi_k)$ converges strongly to $(\phi,\psi)$ in $W^{1,2}\times L^4$.

 Thus, to show the energy identities, it suffices to prove that there exist solutions  $(\sigma_i^l, \xi_i^l)$ of \eqref{EL} with vanishing gravitinos (i.e. Dirac-harmonic maps with curvature term) defined on the standard 2-sphere $\sph^2$, $1\le l\le L_i$, such that
 \begin{equation}
  \begin{split}
   \sum_{i=1}^I \lim_{\delta_i\to 0} \lim_{k\to \infty} E(\phi_k;B_{\delta_i})=\sum_{i=1}^I \sum_{l=1}^{L_i} E(\sigma_i^l), \\
   \sum_{i=1}^I \lim_{\delta_i\to 0} \lim_{k\to \infty} E(\psi_k;B_{\delta_i})=\sum_{i=1}^I \sum_{l=1}^{L_i} E(\xi_i^l).
  \end{split}
 \end{equation}
 This will hold if we prove for each $i=1,\cdots, I$,
 \begin{equation}
  \begin{split}
   \lim_{\delta_i\to 0} \lim_{k\to \infty} E(\phi_k;B_{\delta_i})=\sum_{l=1}^{L_i} E(\sigma_i^l), \\
   \lim_{\delta_i\to 0} \lim_{k\to \infty} E(\psi_k;B_{\delta_i})=\sum_{l=1}^{L_i} E(\xi_i^l).
  \end{split}
 \end{equation}

 First we consider the case that there is only one bubble at the blow-up point $p=p_1$. Then what we need to prove is that there exists a solution $(\sigma^1,\xi^1)$ with vanishing gravitino such that
 \begin{equation}
  \begin{split}
   \lim_{\delta\to 0}\lim_{k\to\infty} E(\phi_k;B_\delta)=E(\sigma^1),  \\
   \lim_{\delta\to 0}\lim_{k\to\infty} E(\psi_k;B_\delta)=E(\xi^1).
  \end{split}
 \end{equation}
 For each $(\phi_k,\psi_k)$, we choose $\lambda_k$ such that
 \begin{equation}
  \max_{x\in D_{\delta}(p)} E\left(\phi_k, \psi_k; B_{\lambda_k}(x)\right)=\frac{\vep_1}{2},
 \end{equation}
 and then choose $x_k\in B_\delta(p)$ such that
 \begin{equation}
  E(\phi_k,\psi_k;B_{\lambda_k}(x_k))=\frac{\vep_1}{2}.
 \end{equation}
 Passing to a subsequence if necessary, we may assume that $\lambda_k\to 0$ and $x_k\to p$ as $k\to\infty$. Denote
 \begin{align}
  \tilde{\phi_k}(x)=\phi_k(x_k+\lambda_k x),  & & \tilde{\psi}_k(x)=\lambda_k^{\frac{1}{2}}\psi_k(x_k+\lambda_k x),
  & & \tilde{\chi}_k=\lambda_k^{\frac{1}{2}}\chi_k(x_k+\lambda_k x).
 \end{align}
 Then $(\tilde{\phi}_k,\tilde{\psi}_k)$ is a solution with respect to $\tilde{\chi}_k$ on the unit disk $B_1(0)$, and by the rescaled conformal invariance of the energies,
 \begin{equation}
  \begin{split}
   &E(\tilde{\phi}_k,\tilde{\psi}_k;B_1(0))=E(\phi_k,\psi_k; B_{\lambda_k}(x_k))=\frac{\vep_1}{2}< \vep_1, \\
   &E(\tilde{\phi}_k,\tilde{\psi}_k;B_R(0))=E(\phi_k,\psi_k; B_{\lambda_k R}(x_k))\le \Lambda.
  \end{split}
 \end{equation}
 Recall that the $\chi_k$'s are assumed to converge in $W^{1,4/3}$ norm. Due to the rescaled conformal invariance in Lemma \ref{conformal transformation lemma}, we have, for any fixed $R>0$,
 \begin{equation}
  \begin{split}
   \int_{B_R(0)} |\tilde{\chi}_k|^4+|\widehat{\na}\tilde{\chi}_k|^{\frac{4}{3}}\dd x
   =\int_{B_{\lambda_k R}(x_k)} |\chi_k|^4+|\widehat{\na}\chi_k|^{\frac{4}{3}}\dv_g\to 0
  \end{split}
 \end{equation}
 as $k\to\infty$. It follows that $\tilde{\chi}_k$ converges to $0$.

 Since we assumed that there is only one bubble, the sequence $(\tilde{\phi}_k,\tilde{\psi}_k)$ strongly converge to some $(\tilde{\phi},\tilde{\psi})$ in $W^{1,2}(B_R,N)\times L^4(B_R,S\times \R^K)$ for any $R\ge 1$. Indeed, this is clearly true for $R\le 1$ because of the small energy regularities, and if for some $R_0\ge 1$, the convergence on $B_{R_0}$ is not strong, then the energies would concentrate at some point outside the unit disk, and by rescaling a second nontrivial bubble would be obtained, contradicting the assumption that there is only one bubble. Thus, since $R$ can be arbitrarily large, we get a nonconstant (because energy $\ge \frac{\vep_1}{2}$) solution on $\R^2$. By stereographic projection we obtain a nonconstant solution on $\sph^2\backslash\{N\}$ with energy bounded by $\Lambda$ and with zero gravitino. Thanks to the removable singularity theorem for Dirac-harmonic maps with curvature term (apply Theorem \ref{removable singularities} with  $\chi\equiv 0$ or see \cite[Theorem 6.1]{jost2015geometric}), we actually have a nontrivial solution on $\sph^2$. This is the first bubble at the blow-up point $p$.

 Now consider the neck domain
 \begin{equation}
  A(\delta, R;k)\coloneqq \{x\in\R^2 |\lambda_k R\le |x-x_k|\le \delta \}.
 \end{equation}
 It suffices to show that
 \begin{equation}\label{no energy on neck}
  \begin{split}
   \lim_{R\to\infty} \lim_{\delta\to 0} \lim_{k\to\infty} E(\phi_k,\psi_k; A(\delta,R;k))=0.
  \end{split}
 \end{equation}
 Note that the strong convergence assumption on $\chi_k$'s implies that
 \begin{equation}\label{limit of gravitinos on necks}
  \lim_{\delta\to 0} \lim_{k\to\infty} \int_{A(\delta,R; k)} |\chi_k|^4+|\widehat{\na}\chi_k|^{\frac{4}{3}} \dd x
  \le\lim_{\delta\to 0} \int_{B_{2\delta(p)}}|\chi|^4+|\widehat{\na}\chi|^{\frac{4}{3}}\dd x=0,
 \end{equation}
 by, say, Lebesgue's dominated convergence theorem.

 To show \eqref{no energy on neck}, it may be more intuitive to transform them to a cylinder. Let $(r_k,\theta_k)$ be the polar coordinate around $x_k$. Consider the maps
 \begin{equation}
  f_k\colon (\R\times\sph^1,(t,\theta), g=\dd t^2+\dd \theta^2)\to (\R^2, (r_k,\theta_k), \dd s^2=\dd r_k^2+r_k^2\dd \theta_k^2)
 \end{equation}
 given by $f_k(t,\theta)=(e^{-t},\theta)$. Then $f_k^{-1}(A(\delta,R;k))=(-\log \delta,-\log\lambda_k R)\times \sph^1\equiv P_k(\delta,R)\equiv P_k$. After a translation in the $\R$ direction, the domains $P_k$ converge to the cylinder $\R\times\sph^1$.  It is known that $f_k$ is conformal
 \begin{equation}
  f_k^*(\dd r_k^2+r_k^2\dd \theta_k^2)=e^{-2t}(\dd t^2+\dd \theta^2).
 \end{equation}
 Thus a solution defined in a neighborhood of $x_k$ is transformed to a solution defined on part of the cylinder via
 \begin{align}
  \Phi_k(x)\coloneqq \phi_k\circ f_k(x), & & \Psi_k(x)\coloneqq e^{-\frac{t}{2}} B\psi_k\circ f_k(x),
  & & X_k(x)\coloneqq e^{-\frac{t}{2}} B\chi_k\circ f_k(x),
 \end{align}
 where $B$ is the isomorphism given in Lemma \ref{conformal transformation lemma}. Note that
 \begin{equation}
  E(\Phi_k,\Psi_k;P_k)=E(\phi_k,\psi_k;A(\delta,R;k)) \le \Lambda,
 \end{equation}
 and that by the remark after Lemma \ref{conformal transformation lemma}, for any $R\in (0,\infty)$,
 \begin{equation}\label{limit behavior of gravitinos}
  \lim_{\delta\to 0} \lim_{k\to\infty} \int_{P_k(\delta,R)} |X_k|^4+|\widehat{\na} X_k|^{\frac{4}{3}} \dd x
  =\lim_{\delta\to 0} \int_{A(\delta,R;k)} |\chi_k|^4+|\widehat{\na}\chi_k|^{\frac{4}{3}}\dd x=0,
 \end{equation}
 which follows from \eqref{limit of gravitinos on necks}.

 For any fixed $T>0$, observe that $(\phi_k,\psi_k,\chi_k)$ converges strongly to $(\phi,\psi,\chi)$ on the annulus domain $B_{\delta}(p)\backslash B_{\delta e^{-T}}(p)$, which implies that $(\Phi_k,\Psi_k,X_k)$ converges strongly to $(\Phi,\Psi,X)$ on $P_T\equiv[T_0, T_0+T]\times\sph^1$, where $T_0=-\log\delta$ and
 \begin{align}
  \Phi(x)\coloneqq \phi\circ f(x), & & \Psi(x)\coloneqq e^{-\frac{t}{2}} B\psi\circ f(x),
  & & X(x)\coloneqq e^{-\frac{t}{2}} B\chi\circ f(x),
 \end{align}
 where $f(t,\theta)=(e^{-t},\theta)$.

 Let $0<\vep<\vep_1$ be given. Because of $E(\phi,\psi)\le \Lambda$ and \eqref{limit behavior of gravitinos}, there exists a $\delta>0$ small such that $E(\phi,\psi;B_\delta(p))<\frac{\vep}{2}$ and such that
 \begin{equation}\label{gravitinos on small disks}
  \int_{B_\delta(x_k)} |\chi_k|^4+|\widehat{\na}\chi_k|^{\frac{4}{3}}\dd x < \frac{\vep}{2}
 \end{equation}
 for large $k$. Thus for the $T$ given above, there is a $k(T)>0$ such that for $k>k(T)$,
 \begin{equation}\label{small end-1}
  E(\Phi_k,\Psi_k;P_T) <\vep.
 \end{equation}
 In a similar way, we denote $T_k\equiv |\log\lambda_k R|$ and $Q_{T,k}\equiv [T_k -T, T_k]\times\sph^1$. Then for $k$ large enough,
 \begin{equation}\label{small end-2}
  E(\Phi_k,\Psi_k; Q_{T,k})<\vep.
 \end{equation}

 For the part in between $[T_0+T, T_k-T]$, we claim that there is a $k(T)$ such that for $k\ge k(T)$,
 \begin{equation}\label{equally smallness}
  \int_{[t,t+1]\times\sph^1} |\na\Phi_k|^2+|\Psi_k|^4 \dd x<\vep, \qquad \forall t\in[T_0, T_k-1].
 \end{equation}
 To prove this claim we will follow the arguments as in the case of harmonic maps in \cite{ding1995energy} and Dirac-harmonic maps in \cite{chen2006dirac}. Suppose this is false, then there exists a sequence $\{t_k\}$ such that $t_k\to\infty$ as $k\to\infty$ and
 \begin{equation}
  \int_{[t_k,t_k+1]\times\sph^1} |\na\Phi_k|^2+|\Psi_k|^4 \dd x\ge \vep.
 \end{equation}
 Because of the energies near the ends are small by \eqref{small end-1} and \eqref{small end-2}, we know that $t_k-T_0, T_k-t_k\to \infty$. Thus by a translation from $t$ to $t-t_k$, we get solutions $(\tilde{\Phi}_k, \tilde{\Psi}_k; \tilde{X}_k)$, and for all $k$ it holds that
 \begin{equation}
  \int_{[0,1]\times\sph^1} |\na\tilde{\Phi}_k|^2+|\tilde{\Psi}_k|^4 \dd x\ge \vep.
 \end{equation}
 From \eqref{limit behavior of gravitinos} we see that $\tilde{X}_k$ go to $0$ in $W^{1,\frac{4}{3}}_{loc}$. Due to the bounded energy assumption we may assume that $(\tilde{\Phi}_k,\tilde{\Psi}_k)$ converges weakly to some~$(\tilde{\Phi}_\infty, \tilde{\Psi}_\infty)$ in $W^{1,2}_{loc}\times L^4_{loc}(\R\times\sph^1)$, passing to a subsequence if necessary. Moreover, by a similar argument as before, the convergence is strong except near at most finitely many points.  If this convergence is strong on $\R\times\sph^1$, we obtain a nonconstant solution with respect to zero gravitino on the whole of $\R\times\sph^1$, hence, by a conformal transformation, a Dirac-harmonic map with curvature term on $\sph^2\backslash\{N,S\}$ with finite energy. The removable singularity theorem then ensures a nontrivial solution on $\sph^2$, contradicting the assumption that $L=1$. On the other hand if the sequence $(\tilde{\Phi}_k,\tilde{\Psi}_k;\tilde{X}_k)$ does not converge strongly to~$(\tilde{\Phi}_\infty, \tilde{\Psi}_\infty; 0)$, then we may find some point $(t_0,\theta_0)$ at which the sequence blows up, giving rise to another nontrivial solution with zero gravitino on $\sph^2$, again contradicting $L=1$. Therefore \eqref{equally smallness} has to hold.

Applying a finite decomposition argument similar to \cite{ye1994gromov, zhao2006energy}, we can divide $P_k$ into finitely many parts
 \begin{align}
  P_k=\bigcup_{n=1}^{\mathcal{N}} P_k^n,& & P_k^n\coloneqq [T^{n-1}_k,T_k^{n}]\times \sph^1, & & T^0_k=T_0, & & T_k^{N}=T_k,
 \end{align}
 where $\mathcal{N}$ is a uniform integer, and on each part the energy of $(\Phi_k,\Psi_k)$ is bounded by $\delta=(\frac{1}{8C_0C_1 C(A)})^2$ where we put $C(A)\coloneqq |A|(|A|+|\na A|)$. Actually, since $E(\Phi_k,\Psi_k;P_k)\le\Lambda$, we know that it can be always divided into at most $\mathcal{N}=[\Lambda/\delta] +1$ parts such that on each part the energy is not more than $\delta$.

 We will use the notation
 \begin{align}
  P^n_k=[T_k^{n-1}, T_k^n]\times\sph^1, & & \bar{P}_k^n=[T_k^{n-1}-1,T_k^n]\times\sph^1,
 \end{align}
 and $\Delta P^n_k=\bar{P}_k^n- P_k^n$. With Lemma \ref{estimates of psi on annulus} on the annuli, we get
 \begin{equation}
  \begin{split}
   \|\Psi_k\|_{L^{4}(P_k^n)}
   &+\|\widetilde{\na}\Psi_k\|_{L^{\frac{4}{3}}(P^n_k)}\\
   \le& C_0\left(|A|\|\na\Phi_k\|_{L^2(\bar{P}_k^n)}+\|QX_k\|^2_{L^4(\bar{P}_k^n)}
               +|A|^2\|\Psi_k\|^2_{L^4(\bar{P}_k^n)}\right) \|\Psi_k\|_{L^4(\bar{P}_k^n)} \\
      &+C\|QX_k\|_{L^4(\bar{P}_k^n)}\|\na\Phi_k\|_{L^2(\bar{P}_k^n)}
       +C\|\Psi_k\|_{L^4(\Delta P_k^n)} \\
      &+C\|\widetilde{\na}\Psi_k\|_{L^{\frac{4}{3}}(T_k^n\times\sph^1)} +C\|\Psi_k\|_{L^4(T_k^n\times\sph^1)}\\
   \le&\frac{1}{4}\|\Psi_k\|_{L^4(P_k^n)}+\frac{1}{4}\|\Psi_k\|_{L^4(\Delta P_k^n)}
        +C\|QX_k\|_{L^4(\bar{P}_k^n)}\|\na\Phi_k\|_{L^2(P_k^n)} \\
      & +C\|QX_k\|_{L^4(\bar{P}_k^n)}\|\na\Phi_k\|_{L^2(\Delta P_k^n)}
        +C\|\Psi_k\|_{L^4(\Delta P_k^n)}\\
      &+C\|\widetilde{\na}\Psi_k\|_{L^{\frac{4}{3}}(T_k^n\times\sph^1)} +C\|\Psi_k\|_{L^4(T_k^n\times\sph^1)},
  \end{split}
 \end{equation}
 where we have used the fact that $\|QX_k\|_{L^4(P_k)}$ can be very small when we take $k$ large and $\delta$ small, because of \eqref{limit behavior of gravitinos}. Note that on $\Delta P_k^n$ the energies of $(\Phi_k,\Psi_k)$ are bounded by $\vep$. Moreover, since on $[T_k^n-1/2, T_k^n+1/2]\times \sph^1$ the small energy assumption holds, thus the boundary terms above are also controlled by $C \vep$ due to the small regularity theorems. Therefore, combining with \eqref{gravitinos on small disks}, we get
 \begin{equation}\label{estimate of Psi_k on P_k}
  \|\Psi_k\|_{L^4(P_k^n)}+\|\widetilde{\na}\Psi_k\|_{L^{\frac{4}{3}}(P^n_k)}
  \le C(\Lambda)\vep^{\frac{1}{4}}.
 \end{equation}
 It remains to control the energy of $\Phi_k$ on $P_k^n$. We divide $P_k^n$ into smaller parts such that on each of them the energy of $\Phi_k$ is smaller than $\vep$. Then the small regularity theorems imply that~$|\phi_k-q_k|\le C_*\sqrt{\vep}$ (which may be assumed to be less than 1), see \eqref{Holder estimate}. Then applying Lemma \ref{estimates of phi on annulus} (transformed onto the annuli) on each small part and summing up the inequalities, one sees that
 \begin{equation}
  \begin{split}
   \int_{P_k^n} |\na\Phi_k|^2\dd x
   \le& C_1 C(A)C_*\sqrt{\vep} \int_{P_k^n} |\na\Phi_k|^2+|\Psi_k|^4+|QX_k|^2|\Psi_k|^2 \dd x  \\
      &  +CC_*\sqrt{\vep}\int_{\p P_k^n} |QX_k||\Psi_k|+|\na\Phi_k| \dd s \\
      &  +C\int_{P_k^n} |\Psi_k|^4+|\widetilde{\na}\Psi_k|^{\frac{4}{3}}+|QX_k|^2|\Psi_k|^2 \dd x.
  \end{split}
 \end{equation}
 Using an argument similar to the above one, and combining with \eqref{estimate of Psi_k on P_k}, we see that
 \begin{equation}
  \int_{P_k^n} |\na\Phi_k|^2 \dd x\le C(\Lambda)\vep^{\frac{1}{3}},
 \end{equation}
 with $C(\Lambda)$ being a uniform constant independent of $k$, $n$, $\mathcal{N}$ and the choice of $\vep$.
 Therefore, on the neck domains,
 \begin{equation}
  \int_{P_k}|\na\Phi_k|^2+|\Psi_k|^4\dd x=\sum_{n=1}^{\mathcal{N}} \int_{P_k^n} |\na\Phi_k|^2+|\Psi_k|^4 \dd x\le C\mathcal{N}\vep^{\frac{1}{3}}.
 \end{equation}
 As $\mathcal{N}$ is uniform (independent of $\vep$ and $k$) and $\vep$ can be arbitrarily small, thus \eqref{no energy on neck} follows, and this accomplishes the proof for the case where there is only one bubble.

 When there are more bubbles, we apply an induction argument on the number of bubbles in a standard way, see \cite{ding1995energy} for the details. The proof is thus finished.

\end{proof}

We remark that the conclusion clearly holds when the gravitino $\chi$ is fixed. Then as Theorem \ref{energy identity thm} shows, a sequence of solutions with bounded energies will contain a weakly convergent subsequence and at certain points this subsequence blows up to give some bubbles. In the language of Teichm\"uller theory \cite{tromba2012teichmuller}, the solution space can be compactified by adding some boundaries, which consists of the Dirac-harmonic maps with curvature term on two-dimensional spheres. This is in particular true when the sequence of gravitinos is assumed to be uniformly small in the $C^1$ norm, which is of interest when one wants to consider  perturbations of the zero gravitinos.

%%%%%%%%%%%%%%%%%%%%%%%%%%%%%%%-------Appendix-----%%%%%%%%%%%%%%%%%%%%%%%%%%%%%%%%%%%%%%%%%%%%%%%%%%%%%%%%%%%%%
\section{Appendix}

In this appendix we show that a weak solution to a system with coupled first and second order elliptic equations on the punctured unit disk can be extended as a weak solution on the whole unit disk, when the system satisfies some natural conditions. This is observed for elliptic systems of second order in the two-dimensional calculus of variations, see \cite[Appendix]{jost1991two}, and we generalize it in the following form.

As before, we denote the unit disk in $\R^2$ by $B_1$ and the punctured unit disk by $B_1^*=B_1\backslash\{0\}$. Let $S$ denote the trivial spinor bundle over $B_1$.

\begin{thm}\label{weak solutions extension}
 Suppose that $\phi\in W^{1,2}(B_1^*, \R^K)$, $\psi\in L^4(B_1^*, S\otimes\R^K)$, $\chi\in L^4(B_1,S\otimes\R^2)$, and they satisfy the system on $B_1^*$
 \begin{equation}\label{mixed system}
  \begin{split}
   \Delta\phi&=F(x,\phi,\na\phi,\psi,\chi)+\diverg_x(V), \\
   \pd\psi   &=G(x,\phi,\na\phi,\psi,\chi),
  \end{split}
 \end{equation}
 in the sense of distributions; i.e. for any $u\in W^{1,2}_0\cap L^\infty(B_1^*,\R^K)$ and any $v\in W_0^{1,\frac{4}{3}}(B_1^*, S\otimes\R^K)$, it holds that
 \begin{equation}
  \begin{split}
   \int_{B_1^*}\langle\na\phi,\na u\rangle\dd x
      &=-\int_{B_1^*} \langle F(x,\phi,\na\phi,\psi,\chi), u\rangle\dd x
          +\int_{B_1^*} \langle V(x,\phi,\na\phi,\psi,\chi), \na u\rangle \dd x, \\
   \int_{B_1^*}\langle\psi,\pd v\rangle \dd x
      &=\int_{B_1^*}\langle G(x,\phi,\na\phi,\psi,\chi),v\rangle\dd x.
  \end{split}
 \end{equation}
 Moreover, assume that the following growth condition is satisfied:
 \begin{equation}\label{growth condition}
  \begin{split}
   |F(x,t,p,q,s)|+|V(x,t,p,q,s)|^2+|G(x,t,p,q,s)|^{\frac{4}{3}} \le C\left(1+|p|^2+|q|^4+|s|^4\right).
  \end{split}
 \end{equation}
 Then for any $\eta\in W_0^{1,2}\cap L^\infty(B_1,\R^K)$ and any $\xi\in W^{1,\frac{4}{3}}_0(B_1, S\otimes\R^K)$, it also holds that
 \begin{equation}\label{weak solution on the whole disk}
  \begin{split}
   \int_{B_1}\langle\na\phi,\na\eta\rangle\dd x
      &=-\int_{B_1} \langle F(x,\phi,\na\phi,\psi,\chi), \eta\rangle\dd x
          +\int_{B_1} \langle V(x,\phi,\na\phi,\psi,\chi), \na\eta\rangle \dd x, \\
   \int_{B_1}\langle\psi,\pd\xi\rangle \dd x
      &=\int_{B_1}\langle G(x,\phi,\na\phi,\psi,\chi),\xi\rangle\dd x.
  \end{split}
 \end{equation}
 That is, when the growth condition \eqref{growth condition} is satisfied, any weak solution to \eqref{mixed system} on the punctured disk $B_1^*$ is also a weak solution on the whole disk.
\end{thm}

\begin{proof}
 For $m\ge 2$, define
 \begin{equation}
  \rho_m(r)=
  \begin{cases}
   1, & \textnormal{for } r\le\frac{1}{m^2}, \\
   \log(1/mr)/\log m, & \textnormal{for } (1/m)^2\le r\le 1/m, \\
   0, & \textnormal{for } r\ge 1/m.
  \end{cases}
 \end{equation}
 Then for any $\eta\in W_0^{1,2}\cap L^\infty(B_1,\R^K)$ and any $\xi\in W^{1,\frac{4}{3}}_0(B_1, S\otimes\R^K)$, set
 \begin{equation}
  \begin{split}
   u_m(x)&=\left( 1-\rho_m(|x|)\right)\eta(x)\in W^{1,2}_0\cap L^\infty(B_1^*,\R^K), \\
   v_m(x)&=\left( 1-\rho_m(|x|)\right)\xi(x)\in  W_0^{1,\frac{4}{3}}(B_1^*, S\otimes\R^K).
  \end{split}
 \end{equation}
 In fact, $|1-\rho_m|\le 1$ and
 \begin{equation}
  \left|\na\rho_m(|x|)\right|=\frac{1}{\log m} \frac{1}{r};
 \end{equation}
 hence
 \begin{equation}
  \int_{B_1}|\na\rho_m(|x|)|^2\dd x=\frac{2\pi}{(\log m)^2}\int_{m^{-2}}^{m^{-1}} \frac{1}{r^2} r\dd r=\frac{2\pi}{\log m}
 \end{equation}
 which goes to 0 as $m\to\infty$. It follows that $u_m\in W^{1,2}_0$. Recalling  the Sobolev embedding in dimension two, $W^{1,\frac{4}{3}}_0(B_1^*)\hookrightarrow L^4(B_1^*)$, $v_m$ lies in $W^{1,\frac{4}{3}}_0(B_1^*)$.

 By assumption,
 \begin{equation}
  \int_{B_1^*}\langle\na\phi,\na u_m\rangle\dd x
      =-\int_{B_1^*} \langle F(x,\phi,\na\phi,\psi,\chi), u_m\rangle\dd x
          +\int_{B_1^*} \langle V(x,\phi,\na\phi,\psi,\chi), \na u_m\rangle \dd x.
 \end{equation}
 Note that $F(x,\phi,\na\phi,\psi,\chi)\in L^1(B_1^*)$ by the growth condition \eqref{growth condition} and $|u_m|\le |\eta|\in L^\infty$. Since $u_m$ converges to $\eta$ pointwisely almost everywhere, thus by Lebesgue's dominated convergence theorem
 \begin{equation}
  \lim_{m\to\infty} \int_{B_1^*} \langle F(x,\phi,\na\phi,\psi,\chi), u_m\rangle\dd x
  =\int_{B_1} \langle F(x,\phi,\na\phi,\psi,\chi), \eta\rangle\dd x.
 \end{equation}
 For the other two terms, note that $\na u_m=-\na\rho_m(|x|)\eta(x)+(1-\rho_m(|x|))\na\eta(x)$. Then
 \begin{equation}
  \left|\int_{B_1^*}\langle\na\phi, -\na\rho_m(|x|)\eta(x)\rangle\right|
  \le \|\na\phi\|_{L^2(B_1^*)}\|\eta\|_{L^\infty(B_1)}\|\na\rho_m\|_{L^2(B_1^*)} \to 0,
 \end{equation}
 as $m\to\infty$, while
 \begin{equation}
  \int_{B_1^*} \langle\na\phi, (1-\rho_m(|x|))\na\eta\rangle\dd x\to \int_{B_1}\langle\na\phi, \na\eta\rangle\dd x
 \end{equation}
 again by Lebesgue's dominated convergence theorem. Thus
 \begin{equation}
  \lim_{m\to\infty} \int_{B_1^*}\langle\na\phi,\na u_m\rangle\dd x
  =\int_{B_1} \langle\na\phi,\na\eta\rangle\dd x.
 \end{equation}
 Similarly
 \begin{equation}
  \lim_{m\to\infty} \int_{B_1^*} \langle V(x,\phi,\na\phi,\psi,\chi),\na u_m\rangle\dd x
  =\int_{B_1} \langle V(x,\phi,\na\phi,\psi,\chi),\na\eta\rangle\dd x.
 \end{equation}
 Therefore, the first equation of \eqref{weak solution on the whole disk} holds.

 Next we show that the second equation of \eqref{weak solution on the whole disk} also holds. Indeed, by assumption
 \begin{equation}
  \int_{B_1^*}\langle\psi,\pd v_m\rangle \dd x
      =\int_{B_1^*}\langle G(x,\phi,\na\phi,\psi,\chi),v_m\rangle\dd x.
 \end{equation}
 Now by the growth condition \eqref{growth condition}, $G(x,\phi,\na\phi,\psi,\chi)\in L^{\frac{4}{3}}(B_1)$, and by Sobolev embedding $\xi\in L^4(B_1)$, thus Lebesgue's dominated convergence theorem implies
 \begin{equation}
  \lim_{m\to\infty} \int_{B_1^*}\langle G(x,\phi,\na\phi,\psi,\chi),v_m\rangle \dd x
  =\int_{B_1} \langle G(x,\phi,\na\phi,\psi,\chi), \xi\rangle\dd x.
 \end{equation}
 On the other hand, $\pd v_m=-\gamma(\na\rho_m(|x|)))\xi+(1-\rho_m(|x|)\pd\xi$, and
 \begin{equation}
  \left|\int_{B_1^*} \langle\psi,-\gamma(\na\rho_m(|x|)))\xi\rangle\dd x\right|
  \le \|\psi\|_{L^4(B_1)}\|\xi\|_{L^4(B_1)}\|\na\rho_m\|_{L^2(B_1)}\to 0,
 \end{equation}
 as $m\to\infty$, while  Lebesgue's dominated convergence theorem implies
 \begin{equation}
  \int_{B_1^*} \langle\psi, (1-\rho_m)\pd\xi\rangle\dd x
  \to \int_{B_1} \langle\psi, \pd\xi\rangle\dd x
 \end{equation}
 since $\pd\xi\in L^{\frac{4}{3}}(B_1)$ and $\psi\in L^4(B_1^*)$. This accomplishes the proof.

\end{proof}

%\section{Appendix}\label{}

% ----------------------------------------------------------------
%\bibliographystyle{amsplain}
%\bibliography{cankaowenxian}

%\providecommand{\bysame}{\leavevmode\hbox to3em{\hrulefill}\thinspace}
%\providecommand{\MR}{\relax\ifhmode\unskip\space\fi MR }
% \MRhref is called by the amsart/book/proc definition of \MR.
%\providecommand{\MRhref}[2]{%
%  \href{http://www.ams.org/mathscinet-getitem?mr=#1}{#2}
%}
%\providecommand{\href}[2]{#2}

\end{document}